\documentclass[10pt]{amsart}
\usepackage[pdftex]{graphicx}
\address[finka@math.berkeley.edu]{Alex Fink, Department of Mathematics,
University of California, Berkeley CA, USA 94720}
\address[b-iriart@uniandes.edu.co]{Benjamin Iriarte Giraldo, 
Departamento de Matem\'aticas,
Universidad de los Andes, Bogot\'a, Colombia}

\numberwithin{equation}{section}

\usepackage{graphicx}

\newtheorem{theorem}{Theorem}[section]
\newtheorem{lemma}[theorem]{Lemma}
\newtheorem{proposition}[theorem]{Proposition}
\newtheorem{corollary}[theorem]{Corollary}

\theoremstyle{definition}
\newtheorem{definition}[theorem]{Definition}
\newtheorem{example}[theorem]{Example}

\theoremstyle{remark}

\newcommand{\Fix}{\mathord{\rm Fix}}
\newcommand{\Part}{\mathord{\rm Part}}
\newcommand{\type}{\mathord{\rm type}}
\newcommand{\id}{\mathord{\rm id}}
\newcommand{\sgn}{\mathop{\rm sgn}}

\newcommand{\mc}{\mathcal}
\newcommand{\lel}{\mathrel{<_{\rm lp}}}
\newcommand{\NC}{^{\rm NC}} 
\newcommand{\NN}{^{\rm NN}}

\title[A bijection between NC and NN partitions]
{A bijection between noncrossing and nonnesting partitions
for classical reflection groups}

\author[Alex Fink and Benjamin Iriarte]{Alex Fink
and Benjamin Iriarte Giraldo}
\thanks{The authors would
like to thank Federico Ardila
and the SFSU-Colombia initiative for their
support in this research.}

\begin{document}
\maketitle

\begin{abstract}
We present an elementary type preserving bijection
between noncrossing and nonnesting partitions for all classical 
reflection groups, answering a question of Athanasiadis.
\end{abstract}

\section{Introduction and background}
The Coxeter-Catalan combinatorics is an active field of study in the 
theory of Coxeter groups.  Several diverse and independently motivated 
sets of objects associated to a Coxeter group $W$ have the cardinality
$\prod_{i=1}^r (h+d_i)/d_i$, where $h$ is the Coxeter number of~$W$
and $d_1,\ldots,d_r$ its degrees.  At the core of the Coxeter-Catalan
combinatorics are the problems of explaining these equalities of cardinalities.
Two of the sets of objects involved are
\begin{itemize}
\item the {\em noncrossing partitions} $NC(W)$, which in their classical
(type~$A$) avatar are a long-studied combinatorial object harking back
at least to Kreweras~\cite{Kreweras}, and in their generalisation to
arbitrary Coxeter groups are due to Bessis and Brady and~Watt
\cite{Bessis}, \cite{BW}; and
\item the {\em nonnesting partitions} $NN(W)$, introduced
by Postnikov~\cite{Postnikov} for all the classical reflection groups simultaneously.
\end{itemize}

Athanasiadis in~\cite{AthNCNN} proved in a case-by-case fashion that
$|NN(W)| = |NC(W)|$ for the classical reflection groups $W$, and asked
for a bijective proof.  This was later improved by 
Athanasiadis and Reiner~\cite{AthRein} to a proof for all
Weyl groups, cited as Theorem~\ref{th:AR} below.
This proof showed that nonnesting and noncrossing partitions
are equidistributed by {\em type}, a statistic for partitions
defined in~\label{def:type}; but it handled the classical 
reflection groups in a nonuniform case-by-case fashion, 
and was not bijective for the exceptional groups.

Our contribution has been to provide a bijection
which, given particular fixed choices of coordinates in the representation, 
works uniformly for the classical reflection groups.
Our proof also provides equidistribution by type.
The cases of our bijection for types $B$, $C$, and~$D$ 
have not appeared before in the literature.  
The ultimate goal in connecting $NN(W)$ and $NC(W)$,
a case-free bijective proof for all Weyl groups, remains open.
The special nature of our choices of coordinates enables
the construction of bump diagrams, and the present lack of a notion of 
bump diagrams for the exceptional groups would seem to be a significant 
obstacle to extending our approach.  

Two other papers presenting combinatorial bijections between noncrossing and
nonnesting partitions independent of this one, 
one by Stump~\cite{Stump} and by Mamede~\cite{Mamede},
appeared essentially simultaneously to it.
Both of these limit themselves to types $A$ and~$B$, 
and our approach is also distinct to them in its type preservation and
in providing additional statistics characterising the new bijections.
More recently Conflitti and Mamede~\cite{CM} 
have presented a bijection in type~$D$ which preserves
different statistics to ours (namely 
{\em openers}, {\em closers}, and {\em transients}).  

In the remainder of this section we lay out the definitions of
the objects involved: in~\S\ref{ssec:uniform}, the uniform definitions of 
nonnesting and noncrossing partitions; in~\S\ref{ssec:classical}, 
a mode of extracting actual partitions from these definitions which
our bijections rely upon; in~\S\ref{ssec:classical NC and NN}, 
the resulting notions for classical reflection groups.
In section~\ref{sec:typepres} we present a type-preserving bijection 
between noncrossing and nonnesting partitions that works for all the classical 
reflection groups. We prove our bijection in a case by case fashion
for each classical type, unpacking and specializing
the definition to a more concrete bijection in each type in turn. 

\subsection{Uniform noncrossing and nonnesting partitions}
\label{ssec:uniform}
For noncrossing partitions we 
follow Armstrong~\cite[\S2.4--6]{Armstrong}.  The treatment of
nonnesting partitions is due to Postnikov~\cite{Postnikov}.

Let $(W,S)$ be a finite Coxeter system of rank~$r$, so that
$S = \{s_1,s_2,\ldots,s_r\}$ generates the group 
$$W = \langle s_1,\ldots,s_r : s_i^2=(s_is_j)^{m_{ij}} = 1\rangle.$$
We will always take the $m_{ij}$ finite.
Let $T = \{s^w:s\in S,w\in W\}$ be the set of all reflections of~$W$,
where $s^w = w^{-1}sw$ denotes conjugation.
Let $[r]=\{1,\ldots,r\}$.  
Consider the $\mathbb R$-vector space 
$V={\rm span}_{\mathbb R}\{\alpha_i : i\in[r]\}$ 
endowed with the
inner product $\langle{\cdot},{\cdot}\rangle$ for which
$\langle\alpha_i,\alpha_j\rangle = -\cos(\pi/m_{ij})$,
and let $\rho:W\to{\rm Aut}(V)$ be the geometric representation of $W$.
This is a faithful representation of $W$, by which it acts
isometrically on~$V$ with respect to $\langle{\cdot},{\cdot}\rangle$.

The set $NC(W)$ of (uniform) noncrossing partitions of~$W$
is defined as an interval of the absolute order.

\begin{definition}
The {\em absolute order} ${\mathrm Abs}(W)$ 
of~$W$ is the partial order on~$W$
such that for $w,x\in W$,  $w\leq x$ if and only if
$$l_T(x) = l_T(w) + l_T(w^{-1}x),$$
where $l_T(w)$ is the minimum length of any
expression for~$w$ as a product of elements of~$T$.  
A word for~$w$ in~$T$ of length $l_T(w)$ will
be called a {\em reduced $T$-word} for~$w$.  
\end{definition}

The absolute order is a poset graded by~$l_T$,
with unique minimal element $1\in W$.  
It has several distinguished maximal elements:

\begin{definition}
A \emph{standard Coxeter element} of $(W,S)$ is any 
element of the form $ c = s_{\sigma(1)}s_{\sigma(2)}\ldots s_{\sigma(r)} $, 
where $ \sigma $ is a permutation of the set $[r]$. 
A \emph{Coxeter element} is any conjugate of a standard Coxeter 
element in~$W$.
\end{definition}

All Coxeter elements have maximal rank in ${\mathrm Abs}(W)$.  

\begin{definition}
Relative to any Coxeter element $ c $, the 
poset of ({\em uniform})
{\em noncrossing partitions} is the interval $ NC(W,c) = [1, c ] $
in the absolute order.
\end{definition}

Although this definition appears to depend 
on the choice of Coxeter element $ c $, the intervals $[1,c]$ 
are isomorphic as posets 
for all $c$~\cite[Definition~2.6.7]{Armstrong}.   
So we are free to use the notation $NC(W)$ for the poset
of noncrossing partitions of~$W$ with respect to any $c$. 

Now assume $W$ is a Weyl group.  
The set $NN(W)$ of nonnesting partitions
is defined in terms of the root poset.
\begin{definition}\label{def:root poset}
The {\em root poset} of~$W$ is its set of positive roots $\Phi^+$ 
with the partial order $\leq$ under which, for $\beta,\gamma\in\Phi^+$,
$\beta\leq\gamma$ if and only if $\gamma-\beta$ lies in the
positive real span of the simple roots.  
\end{definition}
This definition of the root poset is distinct from,
and more suited for connections to nonnesting partitions than, 
the one given in Bj\"orner \& Brenti~\cite{BB}, 
which does not require the Weil group condition, and which in fact
is a strictly weaker order than Definition~\ref{def:root poset}.  

\begin{definition}
A ({\em uniform}) {\em nonnesting partition} 
for~$W$ is an antichain in the root poset of~$W$.  We denote
the set of nonnesting partitions of $W$ by~$NN(W)$. 
\end{definition}

To each root $\alpha$ we have an orthogonal hyperplane $\alpha^\perp$
with respect to $\langle{\cdot},{\cdot}\rangle$, 
and these define a hyperplane arrangement and a poset of intersections.

\begin{definition}
The {\em partition lattice} $\Pi(W)$ of~$W$ is the
intersection poset of reflecting hyperplanes
$$\{\bigcap_{\alpha\in S} \alpha^\perp : S\subseteq\Phi^+\}.$$
\end{definition}

Note that $\Pi(W)$ includes the empty intersection $V$, when $S=\emptyset$.  

Now let $W$ be a {\em classical reflection group}, i.e.\
one of the groups $A_r$, $B_r$, $C_r$ or~$D_r$
in the Cartan-Killing classification.

Each classical reflection group has a standard choice of coordinates
which we will use throughout,
that is an isometric inclusion of~$V$ into 
a Euclidean space $\mathbb R^n$ bearing its usual inner product,
not necessarily an isomorphism.
This yields a faithful isometric representation
$\rho^{\rm cl}:W\to{\rm Aut}(\mathbb R^n)$ of~$W$,
the superscript cl standing for ``classical''.
In Section 2.10 of~\cite{Humphreys},
a standard choice of simple roots is presented in the standard coordinates;
our simple roots, in~\eqref{eq:sys}, are identical except that
we've reversed the indexing, swapping $e_1, e_2, \ldots, e_r$ for
$e_r, e_{r-1},\ldots,e_1$.  
\begin{align}\label{eq:sys}
\Delta_{A_r} &= \left\{e_2-e_1,e_3-e_2,\dots, e_{r+1}-e_r\right\}  \\
\Delta_{B_r} &= \left\{e_1,e_2-e_1,e_3-e_2,\dots, e_r-e_{r-1}\right\}  \nonumber \\
\Delta_{C_r} &= \left\{2e_1,e_2-e_1,e_3-e_2,\dots, e_r-e_{r-1}\right\} \nonumber \\
\Delta_{D_r} &= \left\{e_1+e_2,e_2-e_1,e_3-e_2,\dots, e_r-e_{r-1}\right\} \nonumber
\end{align}
We will reserve $n$ for the dimensions of the particular coordinatisations
presented here, writing $r$ when we mean the rank of~$W$.  
Hence $n=r+1$ when $W=A_r$, but $n=r$ when $W$ is $B_r$ or~$C_r$ or~$D_r$. 
We will use the names $A_{n-1}$, $B_n$, $C_n$, $D_n$ henceforth.

Figure~\ref{fig:rootposets} exhibits the root posets of the classical reflection groups.
We annotate the lower verges of the root posets
with a line of integers, which for reasons of space we bend around the left side.  
Given a dot in Figure~\ref{fig:rootposets}, if $i$ and $j$ are the integers 
in line with it on downward rays of slope 1 and $-1$ respectively, 
then it represents the root $\alpha=e_j-e_i$, where $e_{-k}=-e_k$ for $k<0$ 
and $e_0=0$. 

\begin{figure}[ht]
\centering
\includegraphics{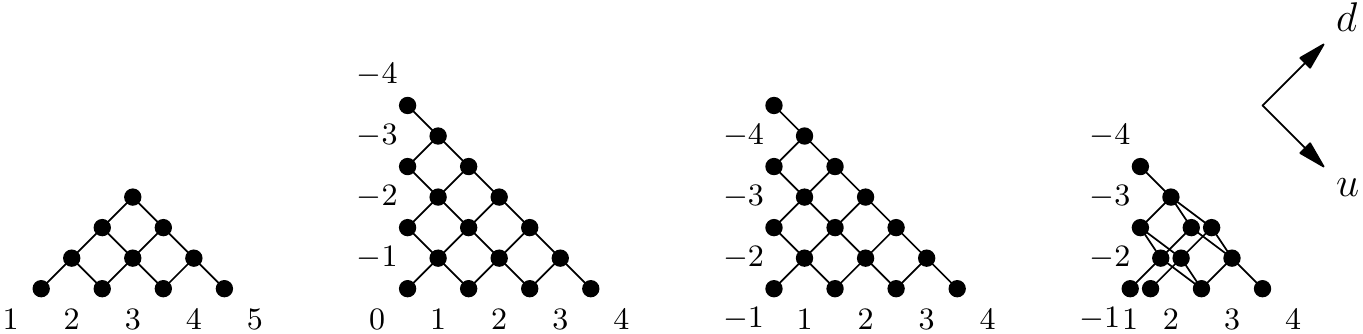}
\caption{The root posets for groups (left to right)
$A_4$, $B_4$, $C_4$, and~$D_4$.}
\label{fig:rootposets}
\end{figure}

\subsection{Classical partitions}\label{ssec:classical}

The definitions of partitions matching the objects considered in classical combinatorics
are framed geometrically in a way that has not been
generalised to all Weyl groups, depending crucially as they do on
the form the reflections take in the standard choice of coordinates.
Our treatment of partitions and our drawings
are taken from Athanasiadis and Reiner~\cite{AthRein}.
We have reversed the orderings of the ground sets from 
Athanasiadis and Reiner's presentation.

Let $W$ be a classical reflection group.  
The procedures to obtain objects representing $NN(W)$
and~$NC(W)$ can be unified to a significant degree --- 
though there will still be cases with exceptional properties ---
so we will speak of classical partitions for~$W$.  
\begin{definition}\label{def:classical partition}
A partition $\pi$ of the set 
$$\Lambda=\{\pm e_i : i=1,\ldots,n\}\cup\{0\},$$
is a {\em classical partition} for $W$
if there exists $L\in\Pi(W)$ such that
each part of~$\pi$ is the intersection 
of~$\Lambda$ with a fiber of the projection to~$\pi$. 
We write $\pi=\Part(L)$.
\end{definition}

We will streamline the notation of classical partitions 
by writing $\pm i$ for $\pm e_i$.  
Thus, a classical partition for $W$ is a partition
of $\pm[n]=\{1,\ldots,n,-1,\ldots,-n,0\}$ for some $n$, 
symmetric under negation.  
A classical partition always contains exactly one part fixed by negation,
which contains the element $0$, namely the fiber over~$0\in L$.
Since the position of 0 is predictable given the other elements,
in many circumstances we will omit it altogether.
If the block containing 0 contains other elements as well, 
we shall call it a {\em zero block}.  Negating all
elements of a block of a classical partition yields a block.
The zero block is the only fixed point of negation, so
the other blocks come in pairs of opposite sign.

For example, a typical classical partition might look like
\begin{equation}\label{eq:typical classical partition}
\{\{1,2\},\{-1,-2\},\{3,-7,-8\},\{-3,7,8\},\{5\},\{-5\},\{4,6,-4,-6,0\}\}
\end{equation}
in which $\{4,6,-4,-6,0\}$ is the zero block.
This is the partition depicted in Figure~\ref{fig:exampleNC}.

\begin{figure}[ht]
\centering
\includegraphics{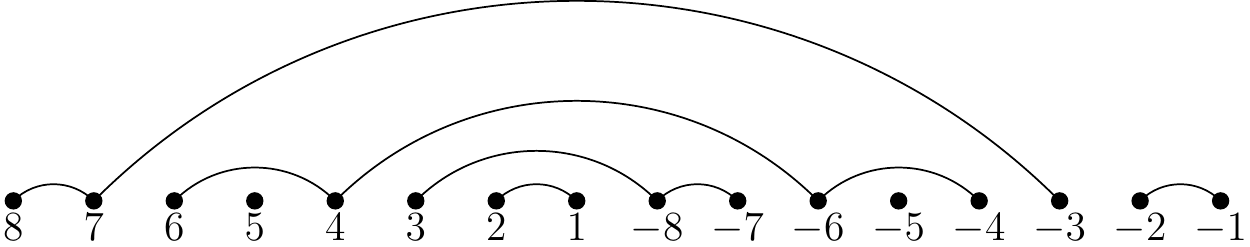}
\caption{Example of a bump diagram of a
noncrossing partition for~$B_8=C_8$.}
\label{fig:exampleNC}
\end{figure}

Given a minimal set of equations for $L$, each of which must be of the form
$$s_1 x_{i_1} = \cdots = s_k x_{i_k} ({}= 0)$$
where the $s_i\in\{+1,-1\}$ are signs,
the classical partition can be read off, one block from each equation.
To the above corresponds
$\{s_1 i_1, \ldots, s_k i_k\}$ if the ${}= 0$ is not included, and 
$\{\pm i_1, \ldots, \pm i_k, 0\}$ if it is.

In case $W=A_{n-1}$, $\rho^{\rm cl}$ fixes the set
of positive coordinate vectors $\{e_i\}$.  So a classical
partition for $W$ will be the union of a partition of $[n]$
and its negative, a partition of $-[n]$, with 0 in a block of its own. 
Here, and in everything we do henceforth with type $A$,
we will omit the redundant nonpositive parts and treat type $A$
partitions as partitions of $[n]$.  

In general the set $\{\pm e_i\}$ is stabilized by $\rho^{\rm cl}$,
giving rise to a faithful permutation representation of $W$.
Combined with the notational efficacies of the
last paragraphs, this is a convenient way to notate elements of~$W$.

To exemplify this notation: for each classical reflection group 
we have a standard choice of
Coxeter element $c$, obtained by taking the product of transpositions
in the order they occur along the bottom of the standard diagram
of the root system.
Using the permutation representations they are
\begin{equation}\label{eq:standard Coxeter elements}
c = \left\{\begin{array}{l@{\quad}l}
(1\ 2\ \ldots\ n) & \mbox{for $W=A_{n-1}$} \\
(1\ \ldots\ n\ (-1)\ \ldots\ (-n)) & \mbox{for $W=B_n=C_n$} \\
(1\ (-1))(2\ \ldots\ n\ (-2)\ \ldots\ (-n)) & \mbox{for $W=D_n$}
\end{array}\right.
\end{equation}

Finally, we introduce the type of a partition.

\begin{definition}\label{def:type}
Let $\pi=\Part(L)$ be a classical partition for a 
classical reflection group $W$.  The {\em type} 
$\type(\pi)$ of~$\pi$ is the 
conjugacy class of $L$ under the action of $W$ on~$\Pi(W)$.
\end{definition}

The collision of terminology between this sense of ``type'' and the
sense referring to a family in the 
Cartan-Killing classification is unfortunate but standard, 
so we muddle along with it.

Combinatorially, the information captured in the type 
of a classical partition is related to the multiset
of its block sizes.  Given a classical partition $\pi$,
let $\lambda$ be the cardinality of its zero block
and $\mu_1,\ldots,\mu_s$ the cardinalities of the pairs of nonzero blocks
of opposite sign.
Then the partitions of the same type as~$\pi$ are exactly those
with zero block of size~$\lambda$ and pairs of other blocks of sizes
$\mu_1,\ldots,\mu_s$.  
The integer partition $\lambda$ which Athanasiadis in~\cite{AthNCNN} 
calls the type of~$\pi$ is the partition $\mu_1,\ldots,\mu_s$.

For example,
a partition has the same type as the partition~\eqref{eq:typical classical partition},
$$\{\{1,2\},\{-1,-2\},\{3,-7,-8\},\{-3,7,8\},\{5\},\{-5\},\{4,6,-4,-6,0\}\},$$
if its zero block of size 4 and it has 
three pairs of nonzero blocks with sizes 3, 2, and~1.

\subsection{Classical noncrossing and nonnesting partitions}\label{ssec:classical NC and NN}

Definitions of the classes of noncrossing
and nonnesting classical partitions are perhaps most intuitively
presented in terms of a diagrammatic representation, motivating
the names ``noncrossing'' and ``nonnesting''.  
After Armstrong~\cite[\S5.1]{Armstrong} we call these {\em bump diagrams}.

Let $P$ be a partition of a totally ordered ground set $(\Lambda,<)$.

\begin{definition}\label{def:bump diagram}
\newcommand{\leP}{\mathbin{<_P}}
Let $G(P)$ be the graph with vertex set $\Lambda$ and edge set 
$$\{(s,s') : \mbox{$s\leP s'$ and $\not\!\exists s''\in S$ s.t.
$s\leP s''\leP s'$}\}$$
where $s\leP s'$ iff $s<s'$ and $s$ and~$s'$ are in
the same block of~$P$.

A {\em bump diagram} of~$P$ is a drawing of $G(P)$ in the plane
in which the elements of $\Lambda$ are arrayed along a horizontal line
in their given order, all edges lie above this line, and
no two edges intersect more than once.
\end{definition}

\begin{definition}\label{def:classical-noncrossing}
$P$ is {\em noncrossing} if its bump diagram contains no
two crossing edges, equivalently if $G(P)$ contains no
two edges of form $(a,c)$, $(b,d)$ with $a<b<c<d$.  
\end{definition}

\begin{definition}\label{def:classical-nonnesting}
$P$ is {\em nonnesting} if its bump diagram contains
no two nested edges, equivalently if $G(P)$ contains
no two edges of form $(a,d)$, $(b,c)$ with $a<b<c<d$.
\end{definition}

The words ``noncrossing'' and ``nonnesting'' perhaps properly
belong as predicates to the bump diagram of~$P$ and not to~$P$ itself,
but we will mostly abuse the terminology slightly and use 
them as just defined.
We will denote the set of classical noncrossing and
nonnesting partitions for $W$ by $NC^{\rm cl}(W)$, resp.\ $NN^{\rm cl}(W)$.  
To define these sets it remains only to specify
the ordered ground set.  

For $NN^{\rm cl}(W)$, the ordering we use is read off the line of integers
in Figure~\ref{fig:rootposets}.  
\begin{definition}\label{def:classNN}
A {\em classical nonnesting partition} for a classical reflection group~$W$ 
is a classical partition for~$W$ nonnesting with respect to the ground set
$$\begin{array}{c@{\quad}l}
1 < \cdots < n & \mbox{if $W=A_{n-1}$;} \\
-n < \cdots < -1 < 0 < 1 < \cdots < n & \mbox{if $W=B_n$;} \\
-n < \cdots < -1 < 1 < \cdots < n & \mbox{if $W=C_n$;} \\
-n < \cdots < -1, 1 < \cdots < n & \mbox{if $W=D_n$.}
\end{array}$$
\end{definition}

A few remarks on the interpretation of these are in order.  

Classical nonnesting partitions for~$B_n$ differ from those for~$C_n$,
reflecting the different root posets.  We have specified that 0 is part of
the ordered ground set for $B_n$.  Despite that, 
per Definition~\ref{def:classical partition},
0 can't occur in a classical partition, it is harmless to consider it present,
coming from the zero vector and forming part of (or perhaps all of) the zero block.
Its presence is quite necessary when drawing bump diagrams:
the dot 0 ``ties down'' a problematic edge of the zero block
in the middle, preventing it from nesting with the others.

The ground set for classical nonnesting partitions for~$D_n$
is not totally ordered but is merely a strict weak ordering, in which
$1$ and $-1$ are incomparable.
Definitions \ref{def:bump diagram} and~\ref{def:classical-nonnesting}
generalise cleanly to this situation,
with no amendments to the text of the definitions themselves.
That is, in a classical nonnesting partition for $D_n$, an edge with $1$ as vertex 
and another with $-1$ as vertex are never considered to nest. 
We diverge in purely cosmetic fashion from Athanasiadis and 
reinforce this last point by aligning these two dots vertically
when drawing a type $D$ nonnesting bump diagram.  

Figure~\ref{fig:exampleNN} exemplifies Definition~\ref{def:classNN},
giving one nonnesting bump diagram for each classical type.

\begin{figure}[ht]
\centering
\includegraphics{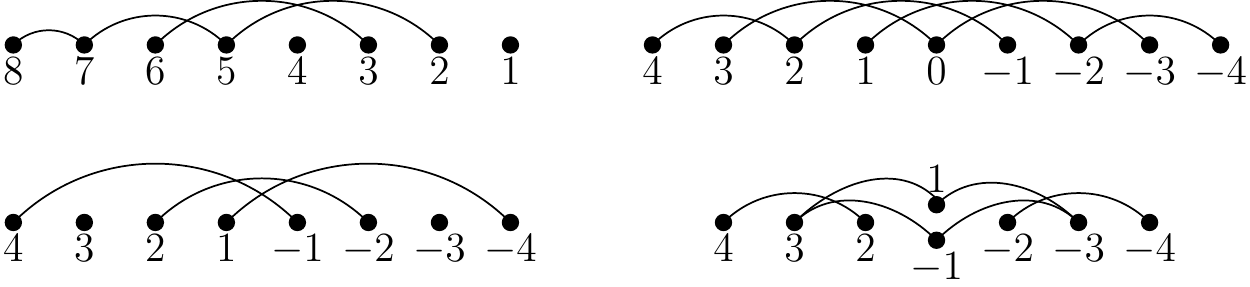}
\caption{Examples of nonnesting bump diagrams in
(top) $A_7$, $B_4$;  (bottom) $C_4$, $D_4$.}
\label{fig:exampleNN}
\end{figure}


For $NC^{\rm cl}(W)$, the ordering we use is read off of the
standard Coxeter elements in~\eqref{eq:standard Coxeter elements}. 
\begin{definition}
A {\em classical noncrossing partition} for a classical reflection group~$W$ not of type~$D$ 
is a classical partition for~$W$ noncrossing with respect to the ground set
$$\begin{array}{c@{\quad}l}
1 < \cdots < n & \mbox{if $W=A_{n-1}$;} \\
-1 < \cdots < -n < 1 < \cdots < n & \mbox{if $W=B_n$;} \\
-1 < \cdots < -n < 1 < \cdots < n & \mbox{if $W=C_n$.} \\
\end{array}$$
\end{definition}
Observe that the order $<$ in these ground sets differs from those for
nonnesting partitions. 

For $D_n$ the standard Coxeter element is not a cycle, so we can't carry
this through, though it's not too far from true that the ground set is
$-2 < \cdots < -n < 2 < \cdots < n$.  
We return to type~$D$ shortly.

{
\newcommand{\lep}{\mathbin{<'}}
These orderings come from cycles, so as one might expect,
if $P$ is noncrossing with respect to $(\Lambda,<)$, it's also
noncrossing with respect to any rotation $(\Lambda,\lep)$ of $(\Lambda,<)$,
i.e.\ any order $\lep$ on~$\Lambda$ given by
$$s\lep t \quad\Leftrightarrow\quad 
\mbox{$t\leq s_0<s$ or $s<t\leq s_0$ or $s_0<s<t$}$$
for some $s_0\in \Lambda$ fixed.
Reflecting this, given any classical partition $P$,
we may bend round the line on which the vertices
of a bump diagram for $P$ lie into a circle, and if we like
supply extra edges for newly adjacent members of the same block,
obtaining a circular bump diagram.  
Then $P$ will be noncrossing if and only if, for every pair
of distinct blocks $B,B'$ of $P$, the convex hulls of the dots
representing $B$ and $B'$ are disjoint.
}
For example, Figure~\ref{fig:exampleNCcirc} 
is the type $B$ or~$C$ noncrossing partition of Figure~\ref{fig:exampleNC}
rendered circularly.

\begin{figure}[ht]
\centering
\includegraphics{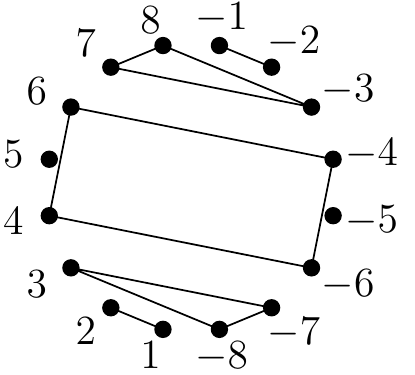}
\caption{The partition of Figure~\ref{fig:exampleNC}
rendered circularly.}
\label{fig:exampleNCcirc}
\end{figure}

The subtleties that occur defining classical noncrossing partitions 
in type~$D$ are significant, and historically it 
proved troublesome to provide the correct notion for this case.
Reiner's first definition~\cite{Reiner} of classical noncrossing
partitions for type~$D$ was later superceded
by that of Bessis and Brady and~Watt~\cite{Bessis}, \cite{BW} and Athanasiadis and
Reiner~\cite{AthRein}, which we use here, for its
better agreement with the uniform definition of~$NC(D_n)$. 
Indeed definitions \ref{def:bump diagram} through~\ref{def:classical-nonnesting} 
require tweaking to handle type~$D$ adequately.  (This said we'll still
use the name ``bump diagram'' for a diagram of a classical 
noncrossing partition for $D_n$.)

\begin{definition}
A {\em classical noncrossing partition} $\pi$ for $D_n$ 
is a classical partition for $D_n$ such that
there exists $c\in\{2,\ldots,n\}$ for which $\pi$ is
noncrossing with respect to both of the ordered ground sets
$$-2<\cdots<-c<-1<-(c+1)<\cdots<-n<2<\cdots<c<1<c+1<\cdots<n$$
and
$$-2<\cdots<-c<1<-(c+1)<\cdots<-n<2<\cdots<c<-1<c+1<\cdots<n$$
The set of these will be denoted $NC^{\rm cl}(D_n)$.
\end{definition}

We will draw these circularly.  Arrange dots labelled 
$-2,\ldots,-n,$ $2,\ldots,n$ 
in a circle and place $1$ and~$-1$ in the middle.
We let $1$ and~$-1$ be drawn coincidently, after~\cite{AthRein}, 
although it would be better to use two circles as in~\cite{KM}, 
with a smaller one in the center on which only $1$ and~$-1$ lie.  
Then a $D_n$ partition $\pi$ is noncrossing if  
and only if no two blocks in this circular bump diagram 
have intersecting convex hulls, except possibly two blocks $\pm B$
meeting only at the middle point.  The edges we will supply in these circular
diagrams are those delimiting the convex hulls of the blocks.
See Figure~\ref{fig:exampleNC D} for an example.

Note that a zero block of precisely two elements 
cannot occur in a classical partition for~$D_n$: a singular equation
$x_i=0$ cannot arise describing a subspace of $\Pi(D_n)$ which is
the intersection of hyperplanes of form $x_i=\pm x_j$.
So the two central dots $\pm1$ belong to different blocks unless they
are both inside the convex hull of some set of vertices among
$\pm\{2,\ldots,n\}$ which are part of the sole zero block.  

\begin{figure}[ht]
\centering
\includegraphics{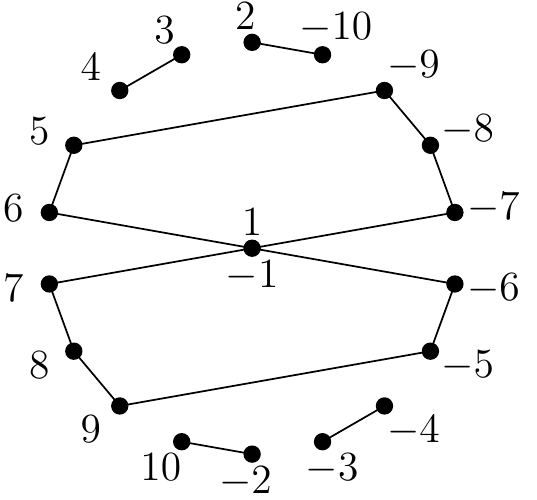}
\caption{Example of a circular bump diagram for a
type $D$ classical nonnesting partition.}
\label{fig:exampleNC D}
\end{figure}

We state without proof the relations
between these classical noncrossing and nonnesting partitions and the uniform ones.  
For $w\in W$, let the {\em fixed space} $\Fix(w)$ of~$w$ be
the subspace of $V(W)$ consisting of vectors fixed by $w$, i.e.\
$\Fix(w) = \ker(w-1)$.

\begin{proposition}\label{prop:utocNC}
The map $f_{NC} : w\mapsto \Part(\Fix(w))$ is a bijection between $NC(W,c)$
and~$NC^{\rm cl}(W)$, where
$c$ is the element in~\eqref{eq:standard Coxeter elements}.
Moreover it is an isomorphism of posets, where
$NC(W,c)$ is given the absolute order 
and $NC^{\rm cl}(W)$ the reverse refinement order.
\end{proposition}

\begin{proposition}\label{prop:utocNN}
The map $f_{NN} : S\mapsto \Part(\bigcap_{\alpha\in S}\alpha^\perp)$
is a bijection between $NN(W)$ and~$NN^{\rm cl}(W)$.  
\end{proposition}

This yields the following elementary descriptions of how to obtain
the edges in a bump diagram.
Starting from an antichain $\pi\in NN(W)$, each root gives
an edge of the nonnesting bump diagram (and its negative), between the two integers
in line with it per the discussion before Figure~\ref{fig:rootposets}.
Starting from a group element $\pi\in NC(W)$, each orbit of the
action of~$\pi$ on~$\{\pm e_i : i=1,\ldots,n\}\cup\{0\}$
gives a block of the noncrossing bump diagram, with an edge between
each element and its image under the permutation representation.

\begin{proposition}\label{th:preliminary}
Consider a
reduced expression in $T$ for some $w\in W$ where $W$ is a Weyl group,
$$w=t_{\alpha_1}t_{\alpha_2}t_{\alpha_3}\dots t_{\alpha_m} \  and \ 
    \alpha_1,\alpha_2,\dots,\alpha_m\in\Phi$$ 
Then $\Fix(w)=\bigcap_{i=1}^m\alpha_i^{\perp}$.
\end{proposition} 

The $\supseteq$ containment clearly holds.  
Them the proposition is an immediate consequence of \cite[Lemma 2]{Carter},
which tells us that the two spaces have the same dimension,


\begin{corollary}\label{cor:intrs}
 Let $\rho$ be a permutation of the set $[m]$. Define
$$w_{\rho}=t_{\alpha_{\rho(1)}}
t_{\alpha_{\rho(2)}}t_{\alpha_{\rho(3)}}\dots t_{\alpha_{\rho(m)}}.$$
Then $\Fix(w)=\Fix(w_{\rho})$. 
\end{corollary} 

So, if we are given an antichain $A$ 
of the root poset for some group~$W$,
we may define $\Fix(A)$ to be $\Fix(\pi_A)$
where $\pi_A$ is the product of the elements 
of~$A$ in any order. 
The elements of an antichain are linearly independent so
Corollary~\ref{cor:intrs} shows that $\Fix(A)$ is well-defined.
See~\cite{Sommers}.

Lastly, the distribution of classical noncrossing and nonnesting partitions 
with respect to type is well-behaved.
In the noncrossing case, the images of the conjugacy classes of the
group $W$ itself are the same as these conjugacy classes of
the action of $W$ on~$\Pi(W)$.  

One can check that 
\begin{proposition}\label{prop:what types are}
Two subspaces $L, L'\in\Pi(W)$ are conjugate if and only if
both of the following hold:
\begin{itemize}
\item the multisets of block sizes $\{|C| : C\in\Part(L)\}$
and $\{|C| : C\in\Part(L')\}$ are equal;
\item if either $\Part(L)$ or~$\Part(L')$
has a zero block, then both do, and these zero blocks have equal size.
\end{itemize}
\end{proposition}
For example, the type $A$ specialisation of this result, 
where zero blocks are irrelevant and we drop the redundant negative elements, 
says that the conjugacy classes of the symmetric group $A_{n-1}$ on $n$ elements
are enumerated by the partitions of the integer $n$.  

We close this section with the statement of the 
equidistribution result of Athana\-sia\-dis and~Reiner~\cite{AthRein}.

\begin{theorem}\label{th:AR}
Let $W$ be a Weyl group.  
Let $f_{NC}$ and~$f_{NN}$ be the functions
of Propositions \ref{prop:utocNC} and~\ref{prop:utocNN}.
For any type~$\lambda$ we have 
$$
|(\type\circ f_{NC})^{-1}(\lambda)| = 
|(\type\circ f_{NN})^{-1}(\lambda)|.
$$
\end{theorem}

\section{A type-preserving bijection for classical groups}\label{sec:typepres}

Throughout this section $W$ will be a classical reflection group.
Partitions will be drawn and spoken of 
with the greatest elements of their ground sets to the left.  

Given any partition, define the order $\lel$ on 
those of its blocks containing positive elements
so that $B\lel B'$ if and only if the least positive element of~$B$
is less than the least positive element of~$B'$.

The notation $NC(W)$ with the Coxeter element omitted will mean $NC(W,c)$,
$c$ being the element in~\eqref{eq:standard Coxeter elements}.

By convention, when we define partition statistics,
we shall observe the convention that Roman letters (like $a$) 
denote ground set elements or tuples thereof, and Greek letters (like $\mu$) 
denote cardinalities or tuples thereof.

\subsection{Statement of the central theorem}\label{ssec:theorem}

\newcommand{\numones}[1]{\#(#1,1)}
\newcommand{\numnegones}[1]{\#(#1,-1)}
\newcommand{\pmtupleset}{\Psi^n}
\newcommand{\llex}{\mathrel{<_{\rm lex}}}
\newcommand{\absdiff}[1]{\left\Arrowvert#1\right\Arrowvert}
\newcommand{\myfans}{f(C_i)}
\newcommand{\myfas}{f(S_j)}

We establish some notation.

\begin{definition}
Let $\pmtupleset$ be the set
of $n$-tuples with entries in $\left\{1,0,-1\right\}$. For any
$u\in\pmtupleset$ define $\numones u$ to be the
number of entries equal to $1$ in $u$ and define
$\numnegones u$ analagously.
Let $\llex$ be the lexicographic order on $n$-tuples.
For any two vectors $a,b\in\mathbb Z^n$, let $\underline{a}$ be
the set of elements of $\mathbb Z^n$ $\llex$-less than
or equal to $a$ and let $\absdiff{a-b}=(|a_1-b_1|,\dots,|a_n-b_n|)$.
\end{definition}   

To any nonnesting or noncrossing partition $x$ of $W$ 
we associate a set $\Omega_x$ which is constructed
inductively with $i$ increasing from
$1$ to~$n$ stepwise. Initially, we begin with
$\Omega_x=\emptyset$.
In step $i$, let $u_i$ be the element of $\pmtupleset\cap \Fix(x)$
with $\absdiff{e_i-u_i}$ $\llex$-minimal (actually 
$\absdiff{e_i-u_i}\in\pmtupleset$ ). 
Whenever $u_i$ is linearly independent with the elements of $\Omega_x$,
let $\Omega_x=\Omega_x\cup\left\{-u_i\right\}$
if $u_i$ has some entry $-1$ and let 
$\Omega_x=\Omega_x\cup\left\{u_i\right\}$ if not. 
Let $\Gamma_x$ be the number of canonical coordinate projections
of $\Fix(x)$ with trivial image $\left\{0\right\}$. 

Lastly, let $E$ be the canonical basis of~$\mathbb R^n$.

\begin{theorem}\label{th:theTheo}
Let $x\in NN(W)$ {\em [resp. $x\in NC(W)$]}. Then, 
there is a unique $y\in NC(W)$ {\em [resp. $y\in NN(W)$]} for which 
$\Gamma_{x}=\Gamma_{y}$ and such that 
the sets $\Omega_x$ and $\Omega_{y}$ are
related to each other in the following way:

There is a bijection $\sigma$ between $\Omega_x$ and
$\Omega_y$ such that for each $u\in \Omega_x$ we have 
$\sigma(u)\in \Omega_y$ satisfying 
\begin{itemize}
\item $\numones u=\numones {\sigma(u)}$ and 
$\numnegones u=\numnegones {\sigma(u)}$ 
\item $|\underline{u}\cap E|=|\underline{\sigma(u)}\cap E|$
\item $|\underline{u}\cap \Omega_x|=
      |\underline{\sigma(u)}\cap \Omega_y|$ 
\item the product of the first two nonzero components of
      $u$ and $\sigma(u)$ is not equal whenever $\numnegones u>1$ and $\numones u>0$
\end{itemize}
Consequently, the induced mapping establishes a 
bijection between noncrossing and nonnesting partitions preserving
orbital type.
\end{theorem}

%

\begin{example}
Let $x$ be the nonnesting partition 
$\{e_2+e_1,e_5-e_1,e_6-e_2,e_8-e_6,e_7-e_3\}$ 
of the group $C_8$. The fixed space $\Fix(x)$
is the following intersection in $\mathbb{R}^8$: 
$$\{v|v_1=-v_2\}\cap\{v|v_1=v_5\}\cap\{v|v_2=v_6\}\cap\{v|v_6=v_8\}
\cap\{v|v_3=v_7\}$$
This is the set 
$\{v\in\mathbb{R}^8|v_1=v_5=-v_2=-v_6=-v_8\mbox{ and }v_3=v_7\}$. 
We can see $\Gamma_x=0$ and also 
$$\Omega_x=\{(-1,1,0,0,-1,1,0,1),(0,0,1,0,0,0,1,0),(0,0,0,1,0,0,0,0)\}$$
Now, we may check 
\begin{multline*}t_{e_7-e_6}t_{e_8-e_7}t_{e_1+e_8}
t_{e_2-e_1}t_{e_5-e_3}t_{e_5-e_2}t_{e_4-e_2}t_{2e_5}
\\=t_{2e_1}t_{e_2-e_1}t_{e_3-e_2}t_{e_4-e_3}
t_{e_5-e_4}t_{e_6-e_5}t_{e_7-e_6}t_{e_8-e_7}\end{multline*}
so $y=t_{e_7-e_6}t_{e_8-e_7}t_{e_1+e_8}
t_{e_2-e_1}t_{e_5-e_3}$ is less than a Coxeter element in the
absolute order and thus is a noncrossing partition
of $C_8$. We can calculate
$\Fix(y)=\{v\in\mathbb{R}^8|v_1=v_2=-v_6=-v_7=-v_8\mbox{ and }
v_3=v_5\}$, so $\Gamma_y=0$ and also
$$\Omega_y=\{(-1,-1,0,0,0,1,1,1),(0,0,1,0,1,0,0,0),
(0,0,0,1,0,0,0,0)\}$$ 
Finally, let $\sigma:\Omega_x\mapsto\Omega_y$ be given
by the assignments in the left column of the following table.  
$$\begin{tabular}{r|c|c|c|c} 
& $\#(u,1)$ & $\#(u,-1)$ & $|\underline{u}\cap E|$ & $|\underline{u}\cap \Omega_x|$ \\[2pt]\hline
\parbox{1.6in}{\hfill$(-1,1,0,0,-1,1,0,1)$\\\mbox{}\hfill$\mapsto (-1,-1,0,0,0,1,1,1)$}
   & 3 & 2 & 0 & 1 \\[2pt]\hline
\parbox{1.6in}{\hfill$(0,0,1,0,0,0,1,0)  $\\\mbox{}\hfill$\mapsto (0,0,1,0,1,0,0,0)  $}
   & 2 & 0 & 6 & 3 \\\hline
\parbox{1.6in}{\hfill$(0,0,0,1,0,0,0,0)  $\\\mbox{}\hfill$\mapsto (0,0,0,1,0,0,0,0)  $}
   & 1 & 0 & 5 & 2 \\\hline
\end{tabular}$$
The remaining columns record the values in Theorem~\ref{th:theTheo}; in each case
they are equally true of~$u$ (and~$\Omega_x$) and of~$\sigma(u)$ (and~$\Omega_y$).
The last bullet in the Theorem only has force in the first line, where it also holds.
So $\sigma$ satisfies the required properties.
\end{example}
In the remainder of the paper we will prove Theorem~\ref{th:theTheo}. 
The four sections that follow will give, in a case by case fashion, the individual
type preserving bijections for each of the classical types
that arise from the theorem.  Then in \S\ref{ssec:proof}
we tie these together and complete the proof.

\subsection{Type $A$}\label{ssec:type A}

The bijection in type $A$, which forms the foundation of the ones for
the other types, is due to Athanasiadis~\cite[\S3]{AthNCNN}.
We include it here to make this foundation explicit and
to have bijections for all the classical groups in one place.

Let $\pi$ be a classical partition for $A_{n-1}$. 
Let $M_1\lel\cdots\lel M_m$ be the blocks of $\pi$, and
$a_i$ the least element of $M_i$, so that 
$a_1<\cdots<a_m$.  Let $\mu_i$ be the cardinality of $M_i$.  
Define the two statistics
$  a(\pi) = (  a_{1}, \ldots,   a_{m}) $ and 
$\mu(\pi) = (\mu_{1}, \ldots, \mu_{m}) $.

It turns out that classical nonnesting and noncrossing partitions 
are equidistributed with respect to these partition statistics, and that
they uniquely determine one partition of either kind. 
This will be the mode in which we present all of our bijections,
which will differ from this one in the introduction of more statistics.

We will say that a list of partition statistics $S$ establishes
a bijection for a classical reflection group $W$ if,
given either a classical noncrossing partition $\pi\NC$ or
a classical nonnesting partition $\pi\NN$ for $W$, 
the other one exists uniquely such that 
$s(\pi\NC) = s(\pi\NN)$ for all $s\in S$.  
We will say it establishes a type-preserving bijection if
furthermore $\pi\NC$ and~$\pi\NN$ always have the same type.


\begin{theorem}\label{th:ir2.3}
The statistics $(a,\mu)$ establish a type-preserving bijection for $A_{n-1}$.
\end{theorem}

The type-preserving assertion in Theorem~\ref{th:ir2.3}
is easy: by Proposition~\ref{prop:what types are} 
the tuple $\mu$ determines the type of any partition that yields it.
As for the bijection itself, we will sketch two different descriptions 
of the process for converting back and forth
between classical noncrossing and nonnesting partitions with the
same tuples $a$, $\mu$, with the intent that they will provide
the reader with complementary suites of intuition.  


To give an example, 
Figures~\ref{fig:Annnc} and~\ref{fig:Ancnn}
show step by step the operation of this bijection in each direction,
in the chain-by-chain fashion of our first proof.
In these figures, the chain $M_i$ being considered appears in bold.  
In the partitions being constructed, 
the elements which are shown with labels and thick dots are those 
less than or equal to the least element $a_i$ of the last placed chain $M_i$.
As we'll see in the proof, no subsequently placed chain can include 
an element less than $a_i$, so these labels are correct.  

\begin{figure}[ht]
\centering
\includegraphics{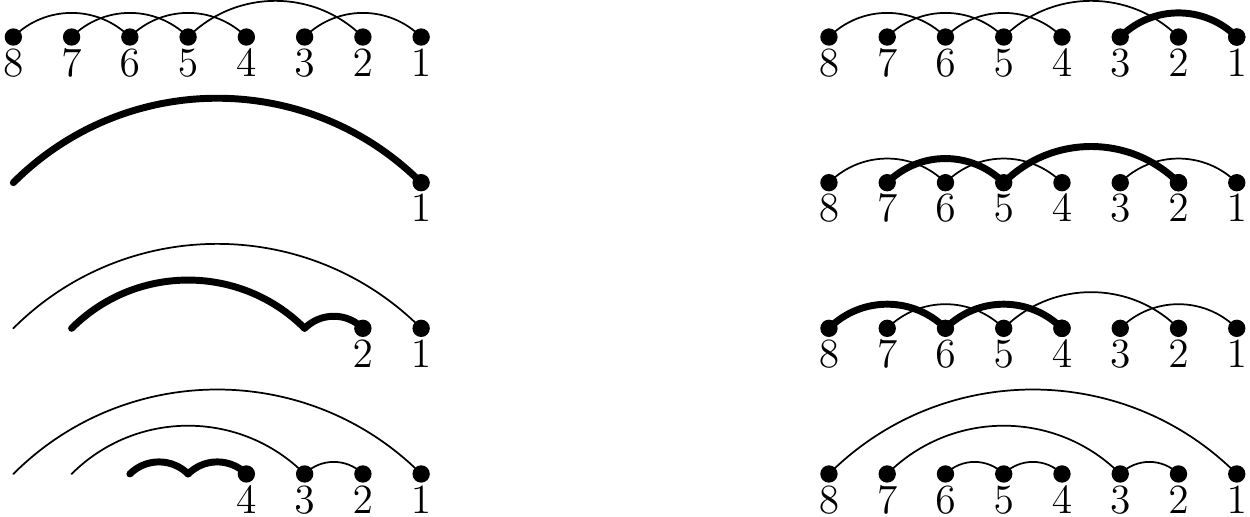}
\caption{The bijection of type $A$ running chain by chain (from left to right,
top to bottom) converting a nonnesting partition to a noncrossing one.
The partitions correspond to $a=(1,2,4)$, $\mu=(2,3,3)$.}
\label{fig:Annnc}
\end{figure}

\begin{figure}[ht]
\centering
\includegraphics{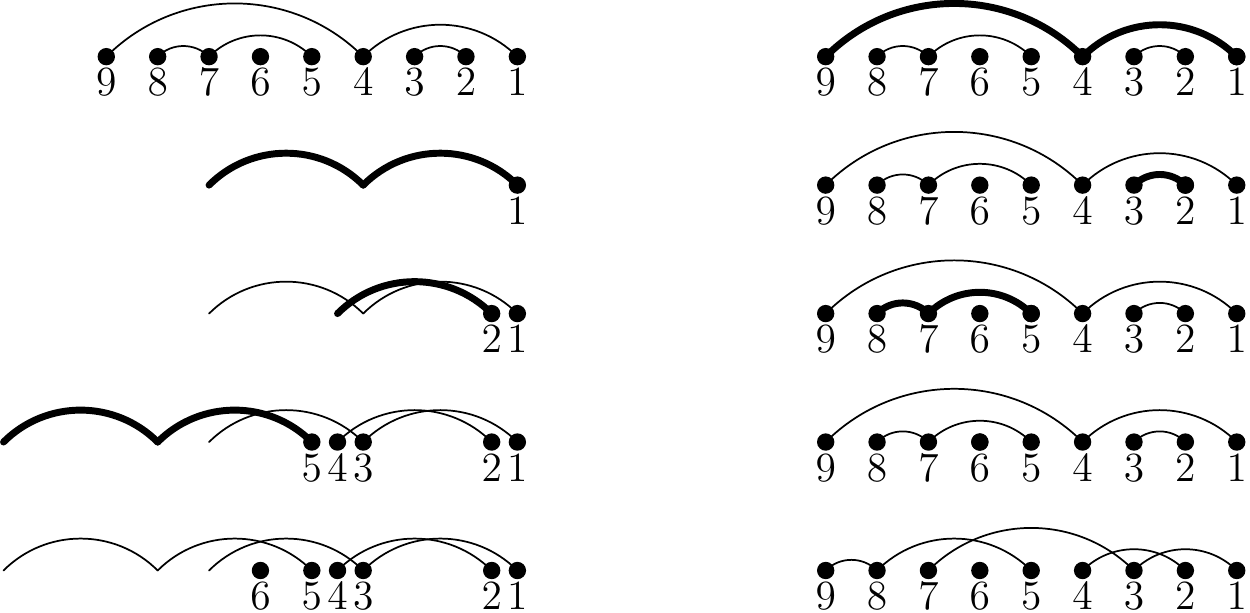}
\caption{The bijection of type $A$ running chain by chain (from left to right,
top to bottom) converting a noncrossing partition to a nonnesting one.
The partitions correspond to $a=(1,2,5,6)$, $\mu=(3,2,3,1)$.}
\label{fig:Ancnn}
\end{figure}

\begin{proof}[Proof of Theorem~\ref{th:ir2.3}: chain by chain]
By a {\em chain} we will mean a sort of incompletely specified block of a 
classical partition, or a connected component of a bump diagram:
a chain has a definite cardinality 
(or {\em length}) but may have unknown elements.
The lengths of the chains of~$\pi$ are determined by~$\mu(\pi)$.
We can view the chains as abstract unlabelled graphs in the plane,
and our task is that of labelling and thereby
positioning the vertices of these chains 
in such a way that the result is nonnesting or noncrossing, as desired.  

To compute the bijection we will inductively {\em place} the 
chains $M_1,\ldots,M_m$, in that order.  When we say a set
$\mc M$ of chains is placed, we mean that all pairwise order 
relations between the elements of the chains in~$\mc M$ are known.  
The effect is that if $\mc M$ is placed, 
we can draw the chains of~$\mc M$ in such a way that 
the bump diagram of any classical partition $\pi$ 
containing blocks whose elements have the order relations of~$\mc M$ 
can be obtained by drawing additional vertices and edges in the bump diagram, 
without redrawing the placed chains.  

Suppose we start with $\pi\NN$ and want to build the 
noncrossing diagram of $\pi\NC$.  Suppose that, for some $j\leq n$, 
we have placed $M_i$ for all $i < j$.  
To place $M_j$, we specify that its least element 
is to be the $a_j$\/th least element among the elements
of all of $M_1,\ldots,M_{j-1},M_j$, and that its remaining elements
are to be ordered in the unique possible way so that 
the placed chains form no crossing.  
In this instance, this means that all the elements of $M_j$ should
be placed consecutively, in immediate succession,
as in Figure~\ref{fig:Annnc}.

To build $\pi\NC$ from~$\pi\NN$ the procedure is the same, except 
that we must order the elements of~$M_j$ in the unique possible way
so that the placed chains form no nesting.  
Concretely, these order relations are the ones we get if every edge is drawn 
with its vertices the same distance apart on the line they lie on,
as in Figure~\ref{fig:Ancnn}.

Note that, in both directions, all the choices we made were unique, 
so the resulting partitions are unique.
\end{proof}

We remark that 
viewing each block of $\pi\NN$ as a chain with a fixed spacing is a particularly
useful picture in terms of the connection between nonnesting partitions and 
chambers of the Shi arrangement~\cite[\S5]{AthNCNN}.

\begin{proof}[Proof of Theorem~\ref{th:ir2.3}: dot by dot]
Let $M_1, \ldots, M_m$ be the blocks of a classical
nonnesting partition $\pi\NN$, such that
the least vertex of $M_i$ is $a_i$.  We describe
an algorithm to build up a
classical noncrossing partition $\pi\NC$ with the same tuples $a$ and $\mu$
by assigning the elements $1,\ldots,n$, in that order, to blocks.

The algorithm maintains a set $\mc O$ of
{\em open blocks}: an open block is a pair $(C,\kappa)$
where $C$ is a subset of the ground set of $\pi\NN$ and 
$\kappa$ a nonnegative integer.  We think of $C$ as a
partially completed block of $\pi\NC$
and $\kappa$ as the number of elements which must be added to $C$
to complete it.  
If $\mc O$ ever comes to contain an open block of form
$(C,0)$, we immediately drop this, for it represents a complete block.
When we begin constructing $\pi\NC$, the set $\mc O$
will be empty.  

Suppose we've assigned the elements $1,\ldots,j-1$ to blocks 
of $\pi\NC$ already, and want to assign $j$.  
If $j$ occurs as one of the $a_i$, then we add a new
singleton block $\{j\}$ to $\pi\NC$ and
add $(\{a_i\},\mu_i-1)$ to~$\mc O$.  
Otherwise, we choose an open block from $\mc O$ according to the
\begin{quote}
{\em Noncrossing open block policy}.  Given $\mc O$, choose
from it the open block $(C,\kappa)$ such that the maximum
element of $C$ is {\em maximal}.  
\end{quote}
We add $j$ to this open block, i.e.\ we replace the block
$C$ of~$\pi\NC$ by $C':=C\cup\{j\}$ and replace 
$(C,\kappa)$ by $(C',\kappa-1)$ within $\mc O$.
The desired partition $\pi\NC$ is obtained after assigning all dots. 

The central observation to make is that this policy
indeed makes $\pi\NC$ noncrossing, and there's a unique way to follow it.
Making a crossing of two edges $(a,c)$ and $(b,d)$, where $a<b<c<d$,
requires assigning $c$ to an open block whose greatest element is then
$a$, when there also exists one with greatest element $b>a$, 
which is witnessed to have been open at the time by its later acquisition of $d$;
this is in contravention of the policy.  

To recover $\pi\NN$ uniquely from $\pi\NC$, the same algorithm 
works, with one modification: instead of the noncrossing
open block policy we use the 
\begin{quote}
{\em Nonnesting open block policy}.  Given $\mc O$, choose
from it the open block $(C,\kappa)$ such that the maximum
element of $C$ is {\em minimal}.  
\end{quote}

This policy makes $\pi\NN$ nonnesting and unique for a similar reason.
If there are nested edges $(a,d)$ and $(b,c)$, where $a<b<c<d$,
then $c$ was added to the block containing $b$ when by policy it 
should have gone with $a$, which was in an open block.
\end{proof}

A careful study of either of these proofs provides a useful characterisation
of the pairs of tuples $a,\mu$ that are the statistics of
a classical nonnesting or noncrossing partition of type $A$.

\begin{corollary}\label{Achar}
Suppose we are given a pair of tuples of positive integers
 $a = (a_1,...,a_{m_1})$, $\mu = (\mu_1,...,\mu_{m_2})$ and 
let $n > 0$. Define $a_0 = 0$ and $\mu_0 = 1$. 
Then, $a$ and $\mu$ represent a classical noncrossing or nonnesting 
partition for $A_{n-1}$ if and only if
\begin{enumerate} 
\item $m_1 = m_2 = m$;
\item $n = \sum_{k=1}^{m} \mu_{k}$ ; and
\item $a_{i-1} < a_i \leq \sum_{k=0}^{i-1} \mu_{k}$ for $i=1,2,...,m$.
\end{enumerate}
\end{corollary}

\subsection{Type $C$}\label{ssec:type C}
In the classical reflection groups other than $A_n$, the negative
elements of the ground set must be treated, and so it will be useful
to have some terminology to deal with these.

\begin{definition}\label{def:block signs}
A {\em positive block} of a classical partition $\pi$ 
is a block of $\pi$ that contains some
positive integer; similarly a {\em negative block} contains a negative integer.  
A \emph{switching block} of $\pi$ is a block of $\pi$ that contains
both positive and nonpositive elements, 
and a {\em nonswitching block} is one that 
contains only positive elements or only nonpositive elements.   

A single edge of the bump diagram is {\em positive} or {\em negative} or
{\em switching} or {\em nonswitching} if it would have
those properties as a block of size 2.  
\end{definition} 

Let $\pi$ be a classical partition for $C_n$.
Given $\pi$, let $M_{1}\lel\cdots\lel M_{m}$ 
be the positive nonswitching blocks of $\pi$, and
$a_i$ the least element of $M_i$.  Let $\mu_i$ be the cardinality of $M_i$.  
These two tuples are reminiscent of type $A$.
Let $P_{1}\lel\cdots\lel P_{k}$ be the switching blocks of 
$\pi$, let $p_{i}$ be the least positive element of $P_{i}$,
and let $\nu_i$ be the number of positive elements of $P_i$.
Define the three statistics
$  a(\pi) = (  a_{1}, \ldots,   a_{m}) $,
$\mu(\pi) = (\mu_{1}, \ldots, \mu_{m}) $,
$\nu(\pi) = (\nu_{1}, \ldots, \nu_{k}) $.
We have 
\begin{equation*}
n =  \sum_{i=1}^{m} \mu_{i} + \sum_{j=1}^{k} \nu_{j}.
\end{equation*}

\begin{theorem}\label{th:ir2.4}
The statistics $(a,\mu,\nu)$ establish a type-preserving bijection for $C_n$.
\end{theorem}

Figure~\ref{fig:typeCbij} illustrates a pair of partitions
related under the resulting bijection.  

\begin{figure}[ht]
\centering
\includegraphics{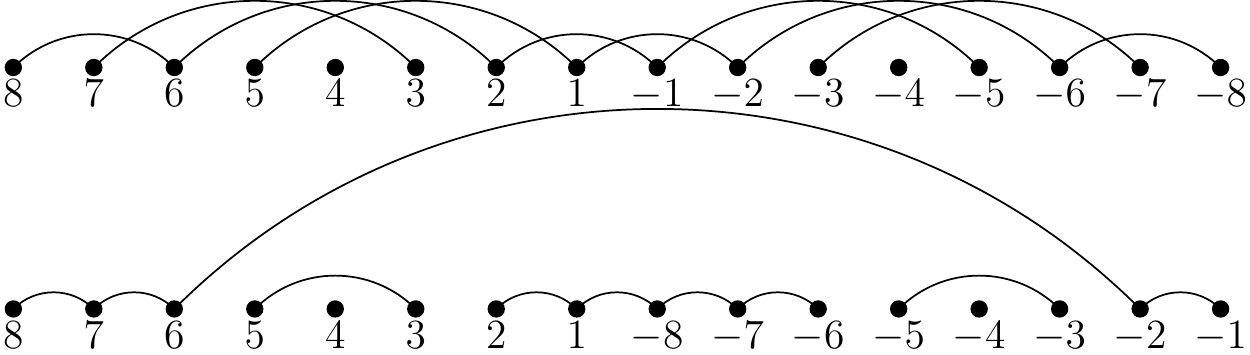}
\caption{The type $C$ nonnesting (top) and noncrossing (bottom)
partitions corresponding to $a=(3,4)$, $\mu=(2,1)$, $\nu=(2,3)$.}
\label{fig:typeCbij}
\end{figure}

\begin{proof}
We state a procedure for converting back and forth
between classical noncrossing and nonnesting partitions 
that preserve the values $a$, $\mu$, and~$\nu$.
Suppose we start with a partition $\pi$, be it 
noncrossing $\pi\NN$ or nonnesting $\pi\NC$, so that we want
to find the partition $\pi'$, being $\pi\NC$ or $\pi\NN$ respectively. 
From $a$, $\mu$, and $\nu$ we inductively construct the positive 
side of~$\pi'$, that is the partition it induces on the set of positive indices $[n]$,
which will determine $\pi'$ by invariance under negation. 

First we describe it from the chain-by-chain viewpoint.
In the bump diagram of $\pi$, consider the labelled 
connected component representing $P_i$, 
which we call the {\em chain} $P_i$.  
Let the ({\em unlabelled}) {\em partial chain} $P_i'$  
be the abstract unlabelled connected graph obtained 
from the chain $P_i$ by removing its 
nonpositive nonswitching edges and nonpositive vertices, leaving the unique
switching edge incomplete, \textit{i.e.}\ 
drawn as a partial edge with just one incident vertex, 
and by dropping the labels.
Notice how
the tuple $\nu$ allows us to draw these partial chains.  
The procedure we followed for type $A$ will generalise to this case,
treating the positive parts of the switching edges first. 

We want to obtain the 
bump diagram for $\pi'$, so we begin by using $\nu$ to
partially draw the chains representing its switching blocks: 
we draw only the positive edges
(switching and nonswitching) of every chain, leaving the unique
switching edge incomplete. 
This is done by reading $\nu$ from back to front and inserting, 
each partial switching chain $P_i'$ in turn
with its rightmost dot placed to the right of
all existing chains, analogously to type $A$.  
In the noncrossing case, we end up with
every vertex of $P_i'$ being strictly to the right of
every vertex of $P_j'$ for $i<j$.  In the nonnesting case,
the vertices of the switching edges will be
exactly the $k$ first positions from right to left 
among all the vertices of $P_1',\dots,P_k'$. 
It remains to place the nonswitching chains $M_{1}, M_{2},\ldots,M_{m}$, 
and this we do also as in the type $A$ bijection, except that at each
step, we place the rightmost vertex of $M_{j}$ so as to become 
the $a_j$\/th vertex, counting from right to left, relative to 
the chains $M_{j-1},\ldots,M_{1}$ and the partial chains
$P_1',P_2',\dots,P_k'$ already placed.

To take the dot-by-dot viewpoint, 
the type $A$ algorithm can be used with only one modification,
namely that $\mc O$ begins nonempty.  It is initialised from $\nu$, as
$$\mc O = \{(\{P_i^-\},\nu_i) : i=1,\ldots,k\},$$ 
where $P_i^-$ is a fictive element that represents the negative elements
of $P_i$ which are yet to be added.  
We must also specify how these fictive elements compare, for use
in the open block policies.  A fictive element is always less than a real element.
In the noncrossing case $P_i^->P_j^-$ iff $i<j$, whereas in the nonnesting
case $P_i^->P_j^-$ iff $i>j$; the variation assures that
$P_1^-$ is chosen first in either case.

Now we have the positive side of~$\pi'$. 
We copy these blocks down again with all parts negated,
and end up with a set of incomplete switching blocks $P_1^*,\ldots,P_k^*$ 
on the positive side and another equinumerous set $-P_1^*,\ldots,-P_k^*$
on the negative side that we need to pair up and connect with edges in
the bump diagram.  

There is a unique way to connect these incomplete blocks to
get the partition~$\pi'$, be it $\pi\NC$ or $\pi\NN$.
In every case $P_i^*$ gets connected with $-P_{k+1-i}^*$, 
and in particular symmetry under negation is attained. 
If there is a zero block it arises from $P_{(k+1)/2}^*$.

Finally, $\pi$ and~$\pi'$ have the same type.
Since the $P_i^*$ are paired up the same way in each, including any zero block,
$\mu$ and $\nu$ determine the multiset of block sizes of 
$\pi$ and~$\pi'$ and the size of any zero block, in identical fashion 
in either case.  Then this is Proposition~\ref{prop:what types are}.
\end{proof} 

Again, a careful look at the preceding proof
gives the characterization of the tuples that describe 
classical noncrossing and nonnesting partitions for type~$C$.
\begin{corollary}\label{Cchar}
Suppose we are given some tuples of positive integers
 $a = (a_1,...,a_{m_1})$, $\mu = (\mu_1,...,\mu_{m_2})$, 
$\nu=(\nu_1,...,\nu_k)$ and 
let $n > 0$. Define $a_0 = 0$ and $\mu_0 = 1$. 
Then, $a$, $\mu$ and $\nu$ represent a classical 
noncrossing or nonnesting partition for $C_n$ if and only if
\begin{enumerate}
\item $m_1 = m_2 = m$;
\item $n =  \sum_{i=1}^{m} \mu_{i} + \sum_{j=1}^{k} \nu_{j}$;
\item $a_{i-1} < a_i \leq \sum_{k=0}^{i-1} \mu_{k}
 + \sum_{j=1}^{k} \nu_{j}$ for $i=1,2,...,m$.
\end{enumerate}
\end{corollary}

\subsection{Type $B$}\label{ssec:type B}
We will readily be able to modify our type~$C$ bijection
to handle type~$B$.  Indeed, if it weren't for our concern about type
in the sense of Definition~\ref{def:type},
we would already possess a bijection for type~$B$, differing from the
type~$C$ bijection only in pairing up the incomplete switching blocks
in a way respecting the presence of the element~0.  Our task is thus
to adjust that bijection to recover the type-preservation.

If $\pi$ is a classical partition for $B_n$, 
we define the tuples $a(\pi)$, $\mu(\pi)$ and $\nu(\pi)$ as in type~$C$. 

Notice that classical noncrossing partitions for $B_n$ and for~$C_n$
are identical, and that the strictly positive part of any 
classical nonnesting partition for~$B_n$ is also
the strictly positive part of some nonnesting $C_{n}$-partition, 
though not necessarily one of the same type. 
Thus Corollary~\ref{Cchar} characterises the 
classical noncrossing or nonnesting partitions for~$B_n$
just as well as for~$C_n$. 

Suppose $\pi$ is a classical nonnesting partition for~$B_n$.
In two circumstances its tuples $a(\pi)$, $\mu(\pi)$, $\nu(\pi)$ 
also describe a unique nonnesting partition for~$C_{n}$ of the same type:
to be explicit, this is when $\pi$ does not contain a zero block, 
and when the unique switching chain in $\pi$ is the one representing the 
zero block. 
If $P_{1}\lel\cdots\lel P_{k}$ are the switching blocks of~$\pi$,
then $\pi$ contains a zero block and more than one switching chain if and only if 
$k$ is odd and $k > 1$.
We notice that $P_{k}$ must be the zero block.  On the other hand, if
$\pi^{\rm C}$ is a classical nonnesting partition for $C_{n}$, the zero block must be
$P_{(k+1)/2}$.  Reflecting this, our bijection will
be forced to reorder $\nu$ to achieve type preservation.

Generalising our prior machinery, we will say that two lists $S\NC$ and~$S\NN$ of
partition statistics, in that order, and a list $\Sigma=(\sigma_i)$ of bijections
establish a (type-preserving) bijection 
for a classical reflection group $W$ if,  
given either a classical noncrossing partition $\pi\NC$ or
a classical nonnesting partition $\pi\NN$ for $W$, 
the other one exists uniquely such that 
$\sigma_i(s_i\NC(\pi\NC)) = s_i\NN(\pi\NN)$
for all $i$ (and furthermore $\pi\NC$ and $\pi\NN$ have the same type).

Suppose we have a tuple $\nu = (\nu_{1},\ldots,\nu_{k}) $ with $k$ odd. 
Define the reordering 
$$\sigma_B(\nu) = 
(\nu_{1},\ldots,\nu_{(k-1)/2},\nu_{(k+3)/2},\ldots,\nu_{k},\nu_{(k+1)/2}).$$ 
If $k$ is not odd then let 
$\sigma_B(\nu) = \nu$.  
Clearly $\sigma_B$ is bijective. For explicitness, the inverse for $k$ odd
is given by
$$\sigma_{B}^{-1}(\nu) = 
(\nu_{1},\ldots,\nu_{(k-1)/2},\nu_{k},\nu_{(k+1)/2},\ldots,\nu_{k-1})$$ 
and for $k$ even $\sigma_B^{-1}(\nu) = \nu$.

\begin{theorem}\label{th:ir2.5}
The lists of statistics $(a,\mu,\nu)$ and $(a,\mu,\nu)$
establish a type-preserving bijection for $B_n$
via the bijections $(\id,\id,\sigma_B)$.
\end{theorem}

\begin{proof}
We use the same procedures as in type~$C$ to convert back and forth
between classical nonnesting and noncrossing partitions, except that
we must rearrange $\nu$ and handle the zero block appropriately,
if it is present. 
When constructing a nonnesting partition
we connect the incomplete switching blocks differently:
in the notation of the dot-by-dot description,
$P_k^*$ must be connected to~$-P_k^*$ and the dot 0,
so that we connect $P_i^*$
to~$-P_{k-i}^*$ for $1 \leq i < k$.  
The conditions of Corollary~\ref{Cchar}, which as we noted above
characterise type~$B$ classical noncrossing and nonnesting partitions,
don't depend on the order of~$\nu$.
So if tuples $a$, $\mu$, and $\nu$ satisfy them then so
do $a$, $\mu$ and~$\sigma_B(\nu)$ (or $\sigma_B^{-1}(\nu)$). 
Thus our statistics establish a bijection between $NC^{\rm cl}(B_n)$ and $NN^{\rm cl}(B_n)$.

Type is preserved, by the definition of $\sigma$ 
and the preceding discussion.
\end{proof}

Figure~\ref{fig:typeBbij} illustrates a pair of partitions
related under the bijection.  

\begin{figure}[ht]
\centering
\includegraphics{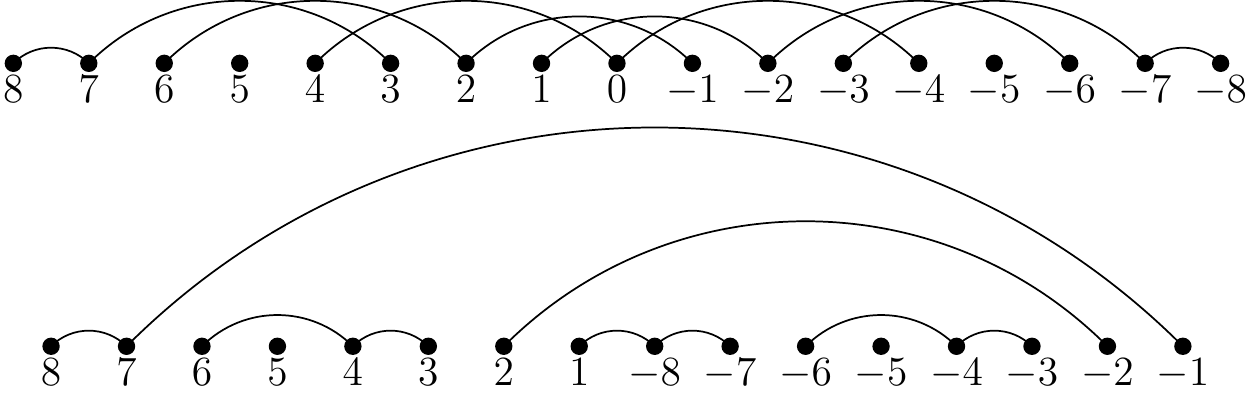}
\caption{The type $B$ nonnesting (top) and noncrossing (bottom)
partitions corresponding to $a = (3,5)$, $\mu=(3,1)$, 
and respectively $\nu=(1,2,1)$ and $\nu=(1,1,2)$.  
Note that $\sigma_B((1,1,2))=(1,2,1)$.
These correspond under the bijection of Theorem~\ref{th:ir2.5}.}
\label{fig:typeBbij}
\end{figure}

\subsection{Type $D$}\label{ssec:type D} 

The handling of type~$D$ partitions
is a further modification of our treatment of the
foregoing types, especially type~$B$. 

In classical partitions for $D_n$, the elements $\pm1$
will play much the same role as the element $0$ of classical nonnesting
partitions for~$B_n$.  So when applying the order $\lel$ and the terminology of
Definition~\ref{def:block signs} in type~$D$ we will regard $\pm1$
as being neither positive nor negative.

Given $\pi\in NN^{\rm cl}(D_n)$,
define the statistics $a(\pi)$, $\mu(\pi)$ and $\nu(\pi)$ as in type $B$.
In this case we must have 
\begin{equation*}
n-1 =  \sum_{i=1}^{m} \mu_{i} + \sum_{j=1}^{k} \nu_{j}.
\end{equation*}
Let $R_{1}\lel\cdots\lel R_{l}$ be the blocks of $\pi$ which 
contain both a positive element and either $1$ and~$-1$.  
It is clear that $l \leq 2$.  
Define the statistic $c(\pi) = (c_{1},\ldots,c_{l})$ by
$c_{i} = R_i\cap\{1,-1\}$.  To streamline the notation
we'll usualy write $c_i$ as one of the symbols $+$, $-$, $\pm$.
Observe that $\pi$ contains a zero block if and only if
$c(\pi) = (\pm)$.  

To get a handle on type $D$ classical noncrossing partitions,
we will transform them into type $B$ ones. 
Let $NC_{\rm r}^{\rm cl}(B_{n-1})$ be a relabelled set of classical
noncrossing partitions for~$B_{n-1}$, in which the parts $1,\ldots,(n-1)$
and $-1,\ldots,-{n-1}$ are changed respectively to $2,\ldots,n$
and $-2,\ldots,-n$.  
Define a map $CM:NC^{\rm cl}(D_n)\to NC_{\rm r}^{\rm cl}(B_{n-1})$, 
which we will call {\em central merging}, 
such that for $\pi\in NN^{\rm cl}(D_n)$,
$CM(\pi)$ is the classical noncrossing $B_{n-1}$-partition 
obtained by first merging the blocks containing $\pm 1$ (which we've
drawn at the center of the circular diagram)
into a single part, and then discarding these elements $\pm 1$.
Define the statistics $a$, $\mu$ and~$\nu$ for$~\pi$ to be equal
to those for $CM(\pi)$, where the entries of $a$ should acknowledge
the relabelling and thus be chosen from $\{2,\ldots,n\}$.  

These statistics do not uniquely characterise $\pi$, so we define additional statistics
$c(\pi)$ and~$\xi(\pi)$.
The definition of $c(\pi)$ is analogous to the nonnesting case:
let $R_{1}\lel\cdots\lel R_{l}$ be the blocks of~$\pi$ which 
intersect $\{1,-1\}$, and define $c(\pi) = (c_{1},\ldots,c_{l})$ where
$c_{i} = R_i\cap\{1,-1\}$.  
Also define $\zeta(\pi) = (\zeta_1,\ldots,\zeta_l)$ where
$\zeta_l = \#(R_l\cap\{2,\ldots,n\})$ is the number of positive
parts of $R_l$.  

Observe that $CM(\pi)$ lacks a zero block if and only if
$c(\pi)=()$, the case that $1$ and $-1$ both belong to
singleton blocks of~$\pi$.  In this case $CM(\pi)$
is just $\pi$ with the blocks $\{1\}$ and $\{-1\}$ removed, 
so that $\pi$ is uniquely recoverable given $CM(\pi)$.  
Otherwise, $CM(\pi)$ has a zero block.  If $c(\pi)=(\pm)$ this
zero block came from a zero block of $\pi$, and $\pi$ is restored
by resupplying $\pm 1$ to this zero block.  Otherwise two blocks of
$\pi$ are merged in the zero block of $CM(\pi)$.  
Suppose the zero block of $CM(\pi)$ is 
$\{c_1,\ldots,c_j,-c_1,\ldots,-c_j\}$, 
with $0<c_1<\cdots<c_j$, so that $j=\sum_{i=1}^l\zeta_l$.  
By the noncrossing and symmetry properties of $\pi$, 
one of the blocks of $\pi$ which was merged into 
this block has the form
$\{-c_{i+1},\ldots,-c_j,c_1,\ldots,c_i,s\}$
where $1\leq i\leq j$ and $s\in\{1,-1\}$.
Then, by definition, 
$c(\pi) = (s,-s)$ and $\xi(\pi) = (i,j-i)$,
except that if $j-i=0$ the latter component of each of these
must be dropped.  In this case the merged blocks of $\pi$ 
can be reconstructed since $c$ and~$\xi$ specify $s$ and~$i$.

Let a {\em tagged} noncrossing partition for~$B_{n-1}$ be an 
element $\pi\in NC_{\rm r}^{\rm cl}(B_{n-1})$
together with tuples
$c(\pi)$ of nonempty subsets of $\{1,-1\}$
and $\xi(\pi)$ of positive integers such that:
\begin{enumerate}
\item the entries of $c(\pi)$ are pairwise disjoint;
\item $c(\pi)$ and $\xi(\pi)$ have equal length;
\item the sum of all entries of~$\xi(\pi)$ is
the number of positive elements in the zero block of $\pi$.
\end{enumerate}

\begin{lemma}\label{lem:central merging}
Central merging gives a bijection between 
classical noncrossing partitions for~$D_n$
and tagged noncrossing partitions for~$B_{n-1}$.
\end{lemma}


\begin{proof}
The foregoing discussion establishes that $CM$ is bijective.
In view of this we need only check 
that the noncrossing property is preserved when
moving between $\pi$ and $CM(\pi)$.  
In terms of bump diagrams, if $CM(\pi)$ is noncrossing $\pi$ is easily seen to be.
For the converse, suppose $CM(\pi)$ has a crossing.  This must be between
the zero block $O$ and some other block $B$ of $CM(\pi)$, so that
it is possible to choose $i,j\in B$ and $k\in O$ such that the
segments $(i,j)$ and $(k,-k)$ within the bump diagram of $CM(\pi)$ cross.
But these segments also cross in the bump diagram for $\pi$ and are contained
within different blocks.
\end{proof}

We show next that partitions
are uniquely determined by the data we have associated with them.

\begin{lemma}
A classical nonnesting partition $\pi$ for $D_n$ is uniquely determined
by the values of $a(\pi)$, $\mu(\pi)$, $\nu(\pi)$, and $c(\pi)$.   
\end{lemma}

\begin{proof} 
We reduce to the analogous facts for classical nonnesting partitions
of types $B$ and~$C$.  There are slight variations in
the behaviour depending on $c(\pi)$, so we break the argument 
into cases.  

If $c(\pi)=()$, then dropping the elements $\pm1$ from $\pi$ 
and relabelling $2,\ldots,n,$ $-2,\ldots,-n$ to $1,\ldots,n-1,-1,\ldots,-(n-1)$
yields a nonnesting partition $\pi'$ for~$C_{n-1}$, and this 
is uniquely characterised by $a(\pi')$, $\mu(\pi')$, and $\nu(\pi')$,
which only differ from the statistics of~$\pi$ by the relabelling in $a$.  

If $c(\pi)=(\pm)$, then merging the elements $\pm1$ into a single
element $0$ and relabelling $2,\ldots,n,-2,\ldots,-n$ to $1,\ldots,n-1,-1,\ldots,-(n-1)$
yields a nonnesting $B_{n-1}$-partition, and this
is again uniquely characterised by $a(\pi')$, $\mu(\pi')$, and $\nu(\pi')$,
which only differ from the statistics of~$\pi$ by the relabelling in $a$. 

The cases $c(\pi)=(-)$ and $c(\pi)=(+,-)$ are carried
under the exchange of $+1$ and $-1$ respectively to
$c(\pi)=(+)$ and $c(\pi)=(-,+)$, 
so it suffices to handle only the latter two.

We claim that, in these latter two cases, 
$\pi$ is itself a classical nonnesting partition for $C_n$.  
We will write $\pi'$ for $\pi$ when we mean to conceive of it
as an element $NN^{\rm cl}(C_n)$; in particular $\pi$ and $\pi'$ 
will have different statistics.
Since the ground set order for $NN^{\rm cl}(C_n)$ is a refinement of the
order for $NN^{\rm cl}(D_n)$ in which only the formerly incomparable elements
$1$ and $-1$ in $\pi$ have become comparable in $\pi'$,
$\pi'$ will be in $NN^{\rm cl}(C_n)$ so long as no nestings
involving edges of $G(\pi)$ terminating at $1$ and $-1$ are introduced.  
By symmetry, if there is such a nesting, there will be one involving
the edges $(i,1)$ and $(j,-1)$ of $G(\pi')$ 
for some $i,j>1$.  But the fact that $c(\pi)$ 
ends with $+$ implies either $i>j$
or the edge $(j,-1)$ does not exist, so there 
is no nesting of this form.  

When we readmit $1$ and $-1$ as positive and negative elements, respectively,
every nonswitching block of $\pi$ remains nonswitching in $\pi'$,
and every switching block of $\pi$ remains switching unless
its only nonpositive element was $1$; in this latter case 
$-1$ is likewise the only nonnegative element of its block, which
happens iff $c(\pi) = (+)$.

Let $a(\pi) = (a_1,\ldots,a_m)$, $\mu(\pi) = (\mu_1,\ldots,\mu_m)$, 
$\nu(\pi) = (\nu_1,\ldots,\nu_m)$.
In the case $c(\pi) = (-,+)$, we have 
\begin{align}\label{eq:DtoC1}
a(\pi') &= a(\pi) \\ 
\mu(\pi') &= \mu(\pi) \notag\\ 
\nu(\pi') &= (\nu_k+1,\nu_1,\ldots,\nu_{k-1}). \notag
\end{align}
That is, the block containing
$1$ is the greatest switching block of~$\pi$ under~$\lel$ by assumption, but
in $\pi'$ where $1$ is positive it
becomes the first switching block.  The other switching blocks are
unchanged in number of positive elements and order, and nothing
changes about the switching blocks.  
In case $c(\pi) = (+)$, 
the block containing $1$ contains no other nonpositive
element, so it becomes a nonswitching block, and in this case we get
\begin{align}\label{eq:DtoC2}
a(\pi') &= (1,a_1,\ldots,a_m) \\
\mu(\pi') &= (\nu_k+1,\mu_1,\ldots,\mu_m) \notag\\
\nu(\pi') &= (\nu_1,\ldots,\nu_{k-1}). \notag
\end{align}
In either case $\pi'$ is a classical nonnesting partition for $C_n$,
and as such is determined by its statistics, but the translations
\eqref{eq:DtoC1} and~\eqref{eq:DtoC2} are injective so that $\pi$
is determined by its statistics as well.
\end{proof}

Note that, when $c(\pi)$ is $(+)$ or $(-,+)$, $\pi'$
is an arbitrary noncrossing partition for~$C_n$
subject to the condition that $1$ is not the only positive element of its block.
The cases $(+)$ and $(-,+)$ can be distinguished 
by whether $a(\pi')$ starts with~1.
Note also that the blocks of $\pi$ which contain one of the parts $\pm1$ are 
exactly those described by the last $l$ components of $\nu(\pi)$, 
where $l$ is the length of~$c(\pi)$.  

All that remains to obtain a bijection
is to describe the modifications to $\nu$ that are needed for correct handling
of the zero block and its components (rather as in type~$B$).  
For a classical nonnesting partition $\pi$ for $D_n$, 
find the tuples $a(\pi)$, $\mu(\pi)$, $\nu(\pi) = (\nu_1,\ldots,\nu_k)$,
and $c(\pi)$.  
Let $\xi(\pi)$ be the tuple of the last $l$ entries of $\nu(\pi)$,
where $l$ is the length of~$c(\pi)$.  Define
\begin{multline*}
\left(\hat\nu(\pi),\xi_{{\rm inv}}(\pi),c_{{\rm inv}}(\pi)\right) 
\\= \left\{\begin{array}{l@{\quad}l}
 \left((\nu_{1},\ldots,\nu_{k/2-1},\nu_{k-1}+\nu_k,\nu_{k/2}\ldots,\nu_{k-2}),
 (\xi_2,\xi_1),(c_2,c_1)\right)
   & \mbox{if $l=2$}  \\ 
 \left((\nu_{1},\ldots,\nu_{(k-1)/2},\nu_{k},\nu_{(k+1)/2}\ldots,\nu_{k-1}),
 \xi(\pi),c(\pi)\right)
   & \mbox{if $l=1$ } \\
 \left(\nu(\pi),\xi(\pi),c(\pi)\right)
   & \mbox{if $l=0$ } \\
\end{array}\right.
\end{multline*}
Define a bijection $\sigma_D$ by $\sigma_D\left(\nu(\pi),c(\pi)\right)= \left(\hat\nu(\pi),\xi_{{\rm inv}}(\pi),c_{{\rm inv}}(\pi)\right)$.
This gives us all the data
for a tagged noncrossing partition $CM(\pi')$ for~$B_{n-1}$, which corresponds
via central merging with a noncrossing partition $\pi'$ for~$D_n$.  
Going backwards, from a noncrossing partition $\pi'$ we recover 
a nonnesting partition $\pi$ by applying central merging, 
finding the list of statistics 
$\left(a(\pi),\mu(\pi),\left(\nu(\pi),c(\pi)\right)\right)$ via the equality
$$\left(\nu(\pi),c_{\pi}\right)=\sigma_D^{-1}\left(\nu(\pi'),\xi(\pi'),
c_(\pi')\right)$$
(the other statistics remain equal) and using these statistics to make a nonnesting
partition as usual. Type preservation
is implied within these modifications of the statistics.
When a zero block exists, the number of positive
parts it contains is preserved because of the equality
$\xi(\pi')=\nu(\pi)_k$ which holds in that case. 
Our handling of $\nu$ 
leaves the components corresponding 
to switching blocks not containing $1$ or $-1$ unchanged, so the
number of positive parts in these blocks is also preserved. The
number of positive parts in the blocks containing $1$ or $-1$ is preserved
because $\xi(\pi')$ corresponds to $c(\pi')$ as the $l$ last entries of
$\nu(\pi)$ correspond to $c(\pi)$  
in all cases. The size of each nonswitching block is preserved in the
statistic $\mu$, as in previous cases.  
 
All in all, we have just proved the following theorem.

\begin{theorem}\label{th:2.6}
The lists of statistics $\left(a,\mu,(\nu,\xi,c)\right)$ and 
$\left(a,\mu,(\nu,c)\right)$ 
establish a type-preserving bijection for $D_n$
via the bijections $(\id,\id,(\sigma_D)^{-1})$.
\end{theorem}

Figures \ref{fig:typeDNN} and~\ref{fig:typeDbij} 
illustrate this bijection. 

\begin{figure}[ht]
\centering
\includegraphics{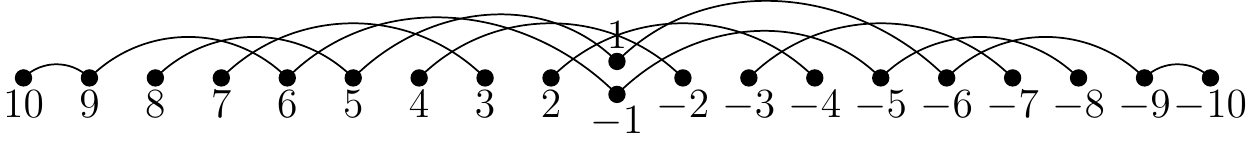}
\caption{The $D_{10}$ nonnesting partition corresponding to 
$a=(3)$, $\mu=(2)$, $\nu=(1,1,2,3)$, $c=(+,-)$ (so $\hat\nu=(1,5,1)$).}
\label{fig:typeDNN}
\end{figure}

\begin{figure}[ht]
\centering
\includegraphics{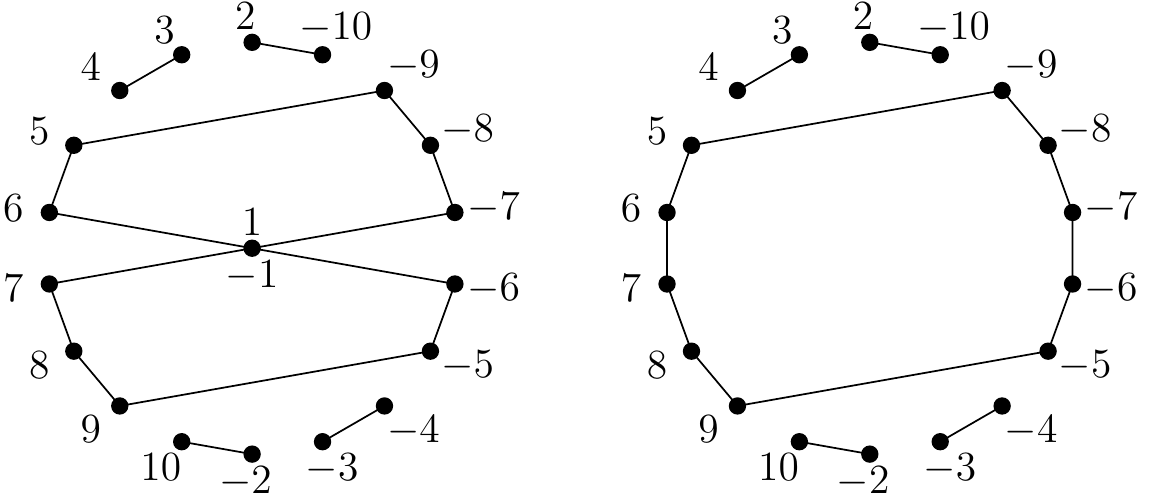}
\caption{ (left) The $D_{10}$ noncrossing partition corresponding to 
$a=(3)$, $\mu=(2)$, $\nu=(1,5,1)$, $\xi = (3,2)$, $c=({-},{+})$. 
(right) The relabelled type $B$ noncrossing 
partition obtained via central merging.  
}
\label{fig:typeDbij}
\end{figure}

Finally we present a characterization of the values of
$a$, $\mu$, $\nu$ and~$c$ that describe type $D$ classical nonnesting partitions.  
As for noncrossing partitions, between our discussion of type $B$
and the definition of tagged partitions and Lemma~\ref{lem:central merging}, 
we have already presented all parts of the analogous result.  

\begin{corollary}\label{Dchar}
Suppose we are given the tuples of positive integers
 $a = (a_1,\ldots,a_{m_1})$, $\mu = (\mu_1,\ldots,\mu_{m_2})$ and 
$\nu=(\nu_1,\ldots,\nu_k)$, and a tuple $c=(c_1,\ldots,c_l)$ 
of nonempty subsets of $\{1,-1\}$.  
Let $n > 0$. Define $a_0 = 1$ and $\mu_0 = 2$. 
Then, $a$, $\mu$, $\nu$ and~$c$ represent a classical nonnesting 
partition for $D_n$ if and only if
\begin{enumerate}
\item $m_1 = m_2 = m$;
\item $n-1 =  \sum_{i=1}^{m} \mu_{i} + 
\sum_{j=1}^{k} \nu_{j}$;
\item  $a_{i-1} < a_i \leq \sum_{k=0}^{i-1} \mu_{k}
 + \sum_{j=1}^{k} \nu_{j}$ for $i=1,2,...,m$;
\item the entries of $c(\pi)$ are pairwise disjoint, so in particular $l\leq 2$;
\item $k-l$ is even.
\end{enumerate}
\end{corollary}

\subsection{Proof of the central theorem}\label{ssec:proof}

Using the preceding bijections we are now ready to establish
our central result.

\newcommand{\AAC}{x^{\rm cl}}
\newcommand{\PIC}{y^{\rm cl}}
\newcommand{\stat}{{\mathcal S}^\ast}
\newcommand{\x}{x^{\rm cl}}
\newcommand{\y}{y^{\rm cl}}

 \begin{proof}[Proof of Theorem~\ref{th:theTheo}]
When defining statistics and 
using the terminology of 
Definition~\ref{def:block signs} we consider positive 
integers as positive elements of
blocks and negative integers as negative ones, without exception.  
Tag these new statistics with $^*$ to distinguish
them from the old statistics defined in
Sections~\ref{ssec:type A} through \ref{ssec:type D}. 
Let $\x$ be the
classical partition representing $x$. Let $\eta^*(\x)$ be the number 
of positive elements
in the zero block of $\x$. 
For any nonzero switching block $P$ of $\x$, define
the {\em joint block} $$S=\displaystyle\min_{\lel}
\left(P,-P\right)$$ and let 
$S_{1}\lel\cdots\lel S_{k'}$
be the joint blocks of $\x$. The number of joint blocks
$k'$ is half the number of nonzero switching blocks.
Let $\vartheta_{+i}^*$ be the number of positive elements in $S_i$ and
let $\vartheta_{-i}^*$ be the number of negative elements in $S_i$ and
define the statistic $\vartheta^*(\x)=
\left((\vartheta_{+1}^*,\vartheta_{-1}^*),\dots,(\vartheta_{+k'}^*,
\vartheta_{-k'}^*)\right)$. Finally, define as usual the statistics 
$a^*(\x)=(a_{1}^*,\dots,a_{m'}^*)$
and $\mu^*(\x)=(\mu_{1}^*,\dots,\mu_{m'}^*)$.

Let $\y$ be the image
of $\x$ under the bijections of Theorems~\ref{th:ir2.3}--\ref{th:2.6}.
There is a simple way to 
find a basis for $\Fix(x)$. From $\x$ define a function $f:\x\rightarrow\pmtupleset$
in the following way.
For any block $B$ of $\x$, let 
$$f(B)=\sgn(B) 
\sum_{b\in B}\frac{b}{|b|}e_{|b|}$$ where $\sgn(B)$ 
is $+1$ or $-1$ so that $f(B)\mathrel{\geq_{\rm lex}}0$ when 
$B$ is nonswitching and $-f(B)\mathrel{\geq_{\rm lex}}0$ when $B$ is 
switching. The set 
$\beta := f\left(\x\right)\setminus\{0\}$ is the
basis we are looking for, which we call the {\em canonical basis} of
$\Fix(x)$. 
 
For a (positive) nonswitching block
$C_i$ of $\x$, we have 
\begin{align}\label{eq:tr1}
|\underline{\myfans}\cap E| &=n+1-a_{i}^* \nonumber \\
\numones \myfans &=\mu_{i}^*  \\ 
\numnegones \myfans &=0 \nonumber \\ 
|\underline{\myfans}\cap \beta|
&=(m'+1-i)+(k') \nonumber 
\end{align}
For a joint block $S_j$ of $\x$, we have
\begin{align}\label{eq:tr2}
|\underline{\myfas}\cap E|&=0 \nonumber \\
\numones \myfas&=\vartheta_{-j}^*  \\ 
\numnegones \myfas&=\vartheta_{+j}^* \nonumber \\ 
|\underline{\myfas}\cap \beta|
&=j \nonumber 
\end{align}
In any case, we have the equality
\begin{equation}\label{eq:tr3}
\Gamma_x=\eta^*(\x)
\end{equation} 
Note that 
$$f(S_1) \llex\dots\llex f(S_{k'})\llex 0 
\llex f(C_{m'})\llex\dots\llex f(C_1)$$ 
and that $m'+k'$ is the number of vectors in 
the ordered basis $\beta$.
In fact 
\begin{equation}
\beta=\left\{f(S_1),\dots, 
f(S_{k'}), 
f(C_{m'}),\dots,f(C_1)\right\} \nonumber
\end{equation}  

Suppose $z$ is nonnesting or noncrossing partition of $W$ and suppose 
$\{v_1,\ldots,v_p\}$ is the canonical basis of $\Fix(z)$, 
ordered so that $v_1\llex\cdots\llex v_n$, which we don't know. 
Let $z^{\rm cl}$ be the classical partition of $z$. 
Then, knowing the statistics 
$\stat:=(a^*,\mu^*,\vartheta^*,\eta^*)$ associated to
$z^{\rm cl}$
allows us to recover the data in
\eqref{eq:tr1} through~\eqref{eq:tr3} associated to each of the
$v_i$, and vice versa.  
    Thus, the first step to reach our goal would be to prove that 
the bijections in Theorems~\ref{th:ir2.3} through \ref{th:2.6} actually
preserve the statistics $\stat$. 
Any of the old statistics for $\y$ that is not mentioned
in the following lines is trivially recovered from $\stat$.

Assume without loss of generality that $x$ is a nonnesting partition, the
other direction being completely analogous. 

We begin with the case where $x$ is a nonnesting partition of $A_{n-1}$. The bijection
of Theorem~\ref{th:ir2.3} clearly preserves $\stat$.
We have $a(\PIC)=a^*(\AAC)$ and $\mu(\PIC)=\mu^*(\AAC)$ so the uniqueness
of $\PIC$ is established directly from the statistics
$\stat$.

Suppose $x$ is an antichain for $C_n$. 
The statistics $a^*$, $\mu^*$ and $\eta^*$ are clearly preserved
in Theorem~\ref{th:ir2.4}. Also $\vartheta_{+i}^*=\nu_i$ and
$\vartheta_{-i}^*=\nu_{k+1-i}$ so $\vartheta^*$ is also
preserved. When there is a zero block we have 
$\nu_{(k+1)/2}=\eta^*$ and this happens if and only if
$\eta^*>0$. Therefore $\PIC$ is characterized 
by $\stat$.
 
Consider the case when $x$ is an antichain for $B_n$.
Again, the statistics $a^*$, $\mu^*$ and $\eta^*$ are clearly preserved
in Theorem~\ref{th:ir2.5}. 
If there is no zero block we have $\vartheta_{+i}^*=\nu_i$ and
$\vartheta_{-i}^*=\nu_{k+1-i}$. 
When there is a zero block we have 
\begin{align}
\vartheta_{+i}^*(\AAC)&=\nu_i(\AAC)&
\vartheta_{+i}^*(\PIC)&=\nu_i(\PIC) \nonumber \\
\vartheta_{-i}^*(\AAC)&=\nu_{k-i}(\AAC)&
\vartheta_{-i}^*(\PIC)&=\nu_{k+1-i}(\PIC) \nonumber
\end{align}
but we also know that
\begin{equation}
\nu_i(\AAC)=\nu_i(\PIC)\qquad\mbox{and}\qquad
\nu_{k-i}(\AAC)=\nu_{k+1-i}(\PIC) \nonumber
\end{equation}
so $\vartheta^*$ is preserved. There is a zero block
if and only if $\eta^*>0$ and here we know in addition
that $\nu_k(\AAC)=\eta^*$ and $\nu_{(k+1)/2}(\PIC)=\eta^*$.
Therefore $\PIC$ is again characterized 
by $\stat$.

\newcommand{\INS}[1]{{\rm Ins}_{#1}}
\newcommand{\DELT}[1]{{\rm Del}_{#1}}

We now consider the case when $x$ is
an antichain for $D_n$.        
This part is divided into several subcases. Consider first when
$c(\AAC)=()$.  
Here, $\AAC$ 
is a classical nonnesting partition for 
$B_n$ and its image $\PIC$ under 
Theorem~\ref{th:2.6} is the unique classical noncrossing partition
from Theorem~\ref{th:ir2.5} 
so the previous type suffices. We know $c(\AAC)=()$ holds exactly
when $a_{1}^*=1$, $\mu_{1}^*=1$ and $\eta^*=0$.

Suppose we
have $c(\AAC)=(+)$. Here the element $+1$ belongs
to a nonswitching block of size $> 1$. In
the bijection of Theorem~\ref{th:2.6} the statistics
$\stat$ are preserved and this case is characterized 
by $a_{1}^*=1$, $\mu_{1}^*>1$ and $\eta^*=0$.
Furthermore, on the noncrossing side we have 
\begin{align*}
a(\PIC)&=\left(\widehat{a}_{1}^*,a_{2}^*,\dots,a_{m'}^*\right)  \\
\mu(\PIC)&=\left(\widehat{\mu}_{1}^*,\mu_{2}^*,\dots,\mu_{m'}^*\right) 
 \\
\nu(\PIC)&=\left(\vartheta_{+1}^*,\dots,\vartheta_{+k'}^*,
\mu_{1}^*-1,\vartheta_{-k'}^*,\dots,\vartheta_{-1}^*\right)  \\
\xi(\PIC)&=(\mu^{*}_1-1)  \\  
c(\PIC)&=(+) 
\end{align*}
so the uniqueness of $\PIC$ is established directly from
$\stat$.   

Suppose $c(\AAC)=(-)$.
The statistics $a^*$, $\mu^*$ and $\eta^*$ are preserved. 
To check that $\vartheta^*$ is preserved we have
\begin{align*}
\vartheta_{+1}^*(\AAC)&=1 & \vartheta_{+1}^*(\PIC)&=1  \\
\vartheta_{-1}^*(\AAC)&=\nu_k(\AAC) & \vartheta_{-1}^*(\PIC)&=
\nu_{\frac{k+1}{2}}(\PIC)  \\
\vartheta_{+i}^*(\AAC)&=\nu_{i-1}(\AAC) & \vartheta_{+i}^*(\PIC)&=
\nu_{i-1}(\PIC)\quad\mbox{ for }i>1  \\
\vartheta_{-i}^*(\AAC)&=\nu_{k-i}(\AAC) & \vartheta_{-i}^*(\PIC)&=
\nu_{k+1-i}(\PIC)\quad\mbox{ for }i>1   
\end{align*} 
However, we know the following equalities hold.
\begin{gather*}
\nu_k(\AAC)=\nu_{\frac{k+1}{2}}(\PIC) \\
\nu_{i-1}(\AAC)=\nu_{i-1}(\PIC) \qquad\mbox{and}\qquad 
\nu_{k-i}(\AAC)=\nu_{k+1-i}(\PIC)
\quad\mbox{ for }i>1
\end{gather*}
Hence, $\vartheta^*$ is indeed preserved.
We also know that $c(\AAC)=(-)$ if and only if
$a_{1}^*>1$, $\vartheta_{+1}^*=1$ and $\eta^*=0$. Using the previous
equations we may see that $\nu(\PIC)$ is obtained
uniquely from $\vartheta^*$, therefore $\PIC$ is
characterized by $\stat$.

Suppose $c(\AAC)=(\pm)$. Here it 
is easily seen that $\stat$ are 
preserved. The characterization for the case is 
$\eta^*>0$ and the uniqueness of $\PIC$ is also
easily established. 

Finally, consider the case when $l=2$ so either
$c(\AAC)=(+,-)$ or $c(\AAC)=(-,+)$ holds. 
To start, suppose that $c(\AAC)=(+,-)$. The bijection of 
Theorem~\ref{th:2.6} preserves $a^*$, $\mu^*$ and $\nu^*$ clearly.
To see that $\vartheta^*$ is also preserved we need the
more intricate equalities
\begin{align*}
\vartheta_{+1}^*(\AAC)&=\nu_{k-1}(\AAC)+1&
\vartheta_{+1}^*(\PIC)&=\xi_{2}(\PIC)+1 \\
\vartheta_{-1}^*(\AAC)&=\nu_{k}(\AAC)&
\vartheta_{-1}^*(\PIC)&=\xi_{1}(\PIC) \\
\vartheta_{+i}^*(\AAC)&=\nu_i(\AAC)&
\vartheta_{+i}^*(\PIC)&=\nu_i(\PIC)\quad\mbox{ for $i>1$} \\
\vartheta_{-1}^*(\AAC)&=\nu_{k-1-i}(\AAC)&
\vartheta_{-1}^*(\PIC)&=\nu_{k+1-i}(\PIC)\quad\mbox{ for $i>1$} 
\end{align*}
But we know from the handling of the statistics for type $D$ that
$$\nu_{k-1}(\AAC)=\xi_{2}(\PIC),$$
because of the function $\sigma_D$;
$$\nu_{k}(\AAC)=\xi_{1}(\PIC),$$
also because of the function $\sigma_D$; and
$$\nu_i(\AAC)=\nu_i(\PIC)\quad\mbox{ and }\quad\nu_{k-1-i}(\AAC)=\nu_{k+1-i}(\PIC)
\quad\mbox{ for $i>1$}.$$
This implies that $\vartheta^*$ is preserved in Theorem~\ref{th:2.6}. 
Note that $c(\AAC)=(+,-)$ or $c(\AAC)=(-,+)$ 
occurs whenever none of the previous cases
holds or whenever $a_{1}^*>1$, $\vartheta_{1}^*>1$ and $\eta^*=0$. Note also
that we can obtain $a(\PIC)$, $\mu(\PIC)$, $\nu(\PIC)$ and
the number of positive and negative elements in the block containing
$+1$ directly from $\stat$,
but we cannot characterize $\PIC$.
This is because the information in $\stat$ does not tell apart
two noncrossing partitions $\PIC_1$ and $\PIC_2$ with identical statistics
$a$, $\mu$ and $\nu$ but such that $\xi(\PIC_1)=\xi_{{\rm inv}}(\PIC_2)$ and
$c(\PIC_1)=c_{{\rm inv}}(\PIC_2)$, $\PIC_1$ and $\PIC_2$ have the 
same statistics $\stat$. In particular 
$\vartheta_{+1}^*(\PIC_1)=\vartheta_{+1}^*(\PIC_2)$ and
$\vartheta_{-1}^*(\PIC_1)=\vartheta_{-1}^*(\PIC_2)$.  
If additionally we require that the element 
with smallest 
absolute value $>1$ in 
the block containing $+1$ 
changes sign from $\AAC$ to $\PIC$, then this would
be Theorem~\ref{th:2.6}. This extends to all 
cases and types in the following way. For a
joint block $S_i$ of $\AAC$ with more than one 
positive element, we require that the element with 
smallest nonminimal absolute value in $S_i$ and the equivalent element in
its image block $S'_i$ of $\PIC$ have opposite signs. 
This new requirement is
simply a necessary condition for noncrossing (or nonnesting) bump diagrams
in all other cases and it was discussed in the proof
of Theorem~\ref{th:ir2.4}, so there is no loss or change in the previous
analysis if we consider it as being part of the bijections. 
However this is tantamount to requiring that
for any such block $S_i$, the product of the first two nonzero components in
$\myfas$ and $f(S'_i)$ is not equal. Clearly
$S_i$ satisfies $\numnegones \myfas>1$ and $\numones \myfas>0$ and
these inequalities are equivalent to the condition imposed on $S_i$.

Therefore, if we can prove that 
$\Omega_x=\beta$ in the general case where $x$ is a nonnesting or
noncrossing partition of $W$,  
we will be done. 

Suppose $x$ is a noncrossing or nonnesting
partition and we are on step $i$ in the construction of $\Omega_x$.
We have $e_i$ and we want to obtain 
$u_i\in\pmtupleset\cap\Fix(x)$ with $\absdiff{u_i-e_i}$
$\llex$-minimal. 

Consider the case when $i$ belongs to a zero block
of $\AAC$. This means that ${\rm \pi}^i\left(\Fix(x)\right)=\left\{0\right\}$
where ${\rm \pi}^i$ is the canonical projection on the $i$-th
coordinate. If $v$ belongs to $\pmtupleset\cap\Fix(x)$
then $\absdiff{v-e_i}\mathrel{\geq_{{\rm lex}}}e_i$ because
$\absdiff{v-e_i}_i=1$. Hence $u_i=0$ and $u_i$ does not
enter $\Omega_x$.

Now consider when $i$ does not belong to a zero block.
We first prove the uniqueness of $u_i$. Suppose there exist
two vectors $u_i$ and $u'_i$ such that $\absdiff{u_i-e_i}$ and
$\absdiff{u'_i-e_i}$ are $\llex$-minimal. This implies
$\absdiff{u_i-e_i}=\absdiff{u'_i-e_i}$ but then 
$(u_i)_i=(u'_i)_i$ and $|(u_i)_j|=|(u'_i)_j|$ for $j\neq i$. 
The $\llex$-minimality condition
implies that $u_i=u'_i=0$ or $(u_i)_i=(u'_i)_i=1$. 
Suppose $(u_i)_i=(u'_i)_i=1$ holds and suppose 
$(u_i)_j=-(u'_i)_j$ for some $j\neq i$. Then 
$(u_i+u'_i)/2$ belongs to $\pmtupleset\cap\Fix(x)$
and $\absdiff{(u_i+u'_i)/2-e_i}\llex\absdiff{u_i-e_i}=
\absdiff{u'_i-e_i}$, a contradiction. Thus $u_i$ is unique. 
Again, the $\llex$-minimality condition implies
$u_i=0$ or $(u_i)_i=1$. If $u_i=0$ then it does not
enter $\Omega_x$. Suppose there exists some $j<i$ such that
$(u_i)_j\neq 0$. In this case $\absdiff{0-e_i}\llex\absdiff{u_i-e_i}$
and we obtain a contradiction. Therefore, if $u_i$ enters $\Omega_x$ then
$(u_i)_j=0$ for $j<i$ and $(u_i)_i=1$. The $\llex$-minimality
condition shows that actually $u_i$ is the vector in 
$\pmtupleset\cap\Fix(x)$ with the least number of nonzero components
such that $(u_i)_j=0$ for $j<i$ and $(u_i)_i=1$. 
Now $u_i$ satisfies these conditions
if and only if $i$ is the least nonzero
component of $f(B)$ for some nonzero block $B$ of $\AAC$
so that $u_i=f(B)$ or $u_i=-f(B)$ (according to whether
$B$ is nonswitching or switching, respectively).
But the sets 
$$S_1,\dots,S_{k'},C_1,\dots,C_{m'}$$ 
are pairwise disjoint and their minimal positive elements 
are all different, and $u_i$ (or $-u_i$) always enters $\Omega_x$,
so we obtain a correspondence between the elements of $\Omega_x$ and
$\beta$.

\end{proof}

\end{document}